\documentclass[reqno]{amsart}
\usepackage[T1]{fontenc}
\usepackage{latexsym,amssymb,amsmath,xcolor}

\usepackage[all,ps]{xy}
\usepackage{amsthm}
\usepackage{longtable}
\usepackage{epsfig}
\usepackage{mathrsfs}
\usepackage{hhline}
\usepackage{epic}
\usepackage{enumerate}
\usepackage{xcolor}
\usepackage{pgf,tikz}
\usepackage{mathrsfs}
\usetikzlibrary{arrows,shapes,calc}

\DeclareMathOperator{\GL}{GL}
\DeclareMathOperator{\sgn}{sgn}

\DeclareMathOperator{\End}{End}

\DeclareMathOperator{\Irr}{Irr}
\DeclareMathOperator{\Res}{Res}
\DeclareMathOperator{\Ind}{Ind}
\DeclareMathOperator{\Gal}{Gal}
\DeclareMathOperator{\Syl}{Syl}
\DeclareMathOperator{\Inf}{Inf}

\newcommand{\tSym}{\widetilde{S}}
\newcommand{\sym}{\mathfrak{S}}
\newcommand{\tAlt}{\widetilde{A}}
\newcommand{\alt}{\mathfrak{A}}

\newcommand{\legendre}[2]{\left(\frac{#1}{#2}\right)}

 \newcommand{\C}{\mathbb{C}}

 \newcommand{\Q}{\mathbb{Q}}
 \newcommand{\Z}{\mathbb{Z}}

\newdimen\shadedBaseline\shadedBaseline=-4mm
\newcount\tableauRow\newcount\tableauCol
\newcommand\ShadedTableau[2][\relax]{%
  \begin{tikzpicture}[scale=0.4,draw/.append style={thick,black},baseline=\shadedBaseline]
    \ifx\relax#1\relax%
    \else 
      \foreach\bx in {#1} { \filldraw[blue!20]\bx+(-.5,-.5)rectangle++(.5,.5); }
    \fi
    \tableauRow=0
    \foreach \Row in {#2} {
       \tableauCol=1
       \foreach\k in \Row {
          \draw(\the\tableauCol,\the\tableauRow)+(-.5,-.5)rectangle++(.5,.5);
          \draw(\the\tableauCol,\the\tableauRow)node{\k};
          \global\advance\tableauCol by 1
       }
       \global\advance\tableauRow by -1
    }
  \end{tikzpicture}%
}
\newcommand\diag[3][\relax]{%
  \begin{tikzpicture}[scale=0.4,draw/.append style={thick,black},baseline=\shadedBaseline]
    \ifx\relax#1\relax%
    \else 
      \foreach\bx in {#1} { \filldraw[blue!20,dashed]\bx+(-.5,-.5)rectangle++(.5,.5); }
    \fi
    \tableauRow=0
    \foreach \Row in {#2} {
       \tableauCol=1
       \foreach\k in \Row {
          \draw(\the\tableauCol,\the\tableauRow)+(-.5,-.5)rectangle++(.5,.5);
          \draw(\the\tableauCol,\the\tableauRow)node{\k};
          \global\advance\tableauCol by 1
       }
       \global\advance\tableauRow by -1
    }
    \foreach \x in {#3} {
      \draw[red,dashed](\x+3,-2-\x)+(-.5,-.5)rectangle++(.5,.5);
    }
  \end{tikzpicture}%
}

\newcommand\frob[7][\relax]{%
  \begin{tikzpicture}[scale=0.4,draw/.append style={black},baseline=\shadedBaseline]
    \ifx\relax#1\relax%
    \else 
      \foreach\bx in {#1} { \filldraw[gray!20,dashed]\bx+(-.5,-.5)rectangle++(.5,.5); }
    \fi
    \tableauRow=0
    \foreach \Row in {#2} {
       \tableauCol=1
       \foreach\k in \Row {
          \draw(\the\tableauCol,\the\tableauRow)+(-.5,-.5)rectangle++(.5,.5);
          \draw(\the\tableauCol,\the\tableauRow)node{\k};
          \global\advance\tableauCol by 1
       }
       \global\advance\tableauRow by -1
    }
        \foreach \x in {#3} {
      \draw[gray!80,dashed](\x+3,-2-\x)+(-.5,-.5)rectangle++(.5,.5);
    }
    \foreach\bx in {#4} {
        \draw[thick]\bx+(-.3,0)rectangle++(.3,.0);
    }
    \foreach\bx in {#5} {
        \draw[thick]\bx+(0,.3)rectangle++(0,-.3);
    }
    \foreach\bx in {#6} {
        \draw[thick,densely dotted]\bx+(-.3,0)--++(.3,.0);
    }
    \foreach\bx in {#7} {
        \draw[thick,densely dotted]\bx+(0,.3)--++(0,-.3);
    }

  \end{tikzpicture}%
}

\newtheorem{theorem}{Theorem}[section] 
\newtheorem{lemma}[theorem]{Lemma}     
\newtheorem{corollary}[theorem]{Corollary}
\newtheorem{proposition}[theorem]{Proposition}

\theoremstyle{definition}
\newtheorem{remark}[theorem]{Remark}
\newtheorem{notation}[theorem]{Notation}

\title[]
{On $\boldsymbol{p}$-rationality in double covers of
symmetric and alternating groups}

\author{Olivier Brunat}

\address{Universit\'e Paris Cit\'e\\ Institut de math\'ematiques de
         Jussieu -- Paris Rive Gauche\\ UFR de math\'e\-matiques\\ Case
7012\\ 75205 Paris Cedex 13\\
         France.}
\email{olivier.brunat@imj-prg.fr}

\author{Rishi Nath}

\address{York College, City University of New York, 
94--20 Guy R. Brewer Blvd. \\
Jamaica, NY 11435\\
USA
}
\email{rnath@york.cuny.edu}


\dedicatory{Dedicated to the memory of Paul Fong and Bhama
Srinivasan}

\begin{document} 
\maketitle

\begin{abstract} 
We construct, for every odd prime $p$, a bijection between the $p'$-spin
characters $\widetilde S_n$ and those of a Sylow normalizer which, in
particular, is equivariant under the Galois action detecting
$p$-rationality. An analogous result holds for $\widetilde A_n$. We also
obtain a blockwise refinement for height-zero spin characters and their
Brauer correspondents. 
%
\end{abstract}

\section{Introduction}

Let $p$ be a prime and let $G$ be a finite group. An irreducible character
of $G$ is called \emph{$p$-rational} if its values lie in a cyclotomic
field $\mathbb Q_m$ for some integer $m$ not divisible by $p$. 

The Galois--McKay conjecture, proposed by G. Navarro
in~\cite{NavarroGalois}, is a Galois refinement of the now-proven McKay
conjecture, whose proof, following more than two decades of work by many
authors, was finally completed by Cabanes and Sp\"ath \cite{CabSpath}. 
It predicts the existence of a bijection between the irreducible
characters of $p'$-degree of $G$ and those of $\operatorname{N}_G(P)$,
where $P$ is a Sylow $p$-subgroup of $G$, which is equivariant under the
action of a certain subgroup $\mathcal H$ of the absolute Galois group
$\operatorname{Gal}(\overline{\mathbb{Q}}/\mathbb{Q})$. 
In particular, one consequence of this conjecture is that the numbers of
$p$-rational characters of $p'$-degree agree on the global and local
sides.

Recent progress on Galois refinements of the McKay conjecture includes the
reduction theorem of Navarro--Sp\"ath--Vallejo~\cite{NavarroSpaethVallejo}
and, more recently, the proof by Ruhstorfer--Schaeffer
Fry~\cite{RuhstorferSchaefferFry} of the Isaacs--Navarro Galois
conjecture~\cite{IsaacsNavarro2002} . The latter involves a subgroup
$\mathcal H_0$ of $\mathcal H$ and yields an $\mathcal H_0$-equivariant
bijection. For odd primes, however, $\mathcal H_0$ is strictly contained
in the subgroup $G(p)$ detecting $p$-rationality, so this result does not
by itself give a correspondence preserving $p$-rationality.
For alternating groups, the authors construted in  \cite{BrNa} a
Galois-equivariant McKay correspondence. More recently, they studied the
Galois action on spin characters and its compatibility with Littlewood
decompositions \cite{BrNaMZ}.

\medskip
The aim of this paper is to study this phenomenon for spin characters of
double covers of symmetric and alternating groups. We write
\[
1\longrightarrow\langle z\rangle
\longrightarrow\widetilde S_n
\longrightarrow S_n
\longrightarrow1
\]
for a double cover of $S_n$, and a character of $\widetilde S_n$ is
called spin when $z$ acts as $-1$. The sign character of $S_n$ lifts
to a linear character $\varepsilon$ of $\widetilde S_n$, and tensoring
by $\varepsilon$ defines the usual association of spin characters.
An irreducible character whose degree is prime to $p$ will be called a
$p'$-character.
For any subgroup $G\leq \widetilde S_n$ containing $z$, we write
$\operatorname{Irr}_{p'}^{\mathrm{spin}}(G)$ for its irreducible spin
$p'$-characters.
\smallskip 

Let $G(p)$ denote the subgroup of $\mathcal H$ consisting of those
automorphisms which act trivially on all roots of unity of order prime to
$p$. Then an irreducible character is $p$-rational if and only if it is
fixed by $G(p)$.
Let $P\in\operatorname{Syl}_p(\widetilde S_n)$ and set $N=N_{\widetilde
S_n}(P)$. We prove the following result.
\begin{theorem}
\label{thm:main} Assume that $p$ is an odd prime.
There exists a bijection
\[
\Omega:
\operatorname{Irr}_{p'}^{\mathrm{spin}}(\widetilde S_n)
\longrightarrow
\operatorname{Irr}_{p'}^{\mathrm{spin}}(N)
\]
which is compatible with association and $G(p)$-equivariant. In
particular, the numbers of $p$-rational spin characters of $p'$-degree of
$\widetilde S_n$ and of $N$ are equal.
\end{theorem}
An analogous result will also be obtained for the double cover $\widetilde
A_n$ of the alternating group.
\smallskip

We will use the following strategy. The main point is that the global and
local characters are parametrized by the same combinatorial data, namely
the \emph{$ p$-bar core towers} of bar partitions of $n$ (that
is, partitions with distinct parts). On the global side, irreducible spin
characters of  $\widetilde S_n$ are classically labelled by bar
partitions $\lambda$, or equivalently by their $ p$-bar core
towers.
Each $\lambda$ labels either one self-associate spin character or a pair
of associated spin characters. On the local side, the work of
Michler--Olsson~\cite{MichlerOlsson} implicitly yields the same
combinatorial description. We make it explicit and adapt their local
framework so that both association and the Galois action can be followed
on the two sides.

We also point out that the local construction in Michler--Olsson contains
a gap which goes beyond minor typographical errors. Their construction
does not cover all the representations needed for the local
parametrization.
We correct this issue here by extending their approach to the missing
representations. This correction may be of independent interest.

\medskip

The paper is organized as follows. In Section~\ref{sec:part1}, we recall
the parametrization and character values of the spin characters of the
double covers of the symmetric and alternating groups, describe their
Galois action, and relate $p$-rationality to the $p$-bar core tower.
In Section~\ref{sec:part2}, we consider the local case.
To study the irreducible characters of $p'$-degree on the local side, we
first reduce the Sylow normalizer by the derived subgroup of its Sylow
$p$-subgroup. We then describe the spin irreducible representations of
the resulting quotient one $p$-adic level at a time. For a fixed
level $\beta$, we first determine the elementary spin representations of
$N_\beta^+$ and then construct the representations attached to their
homogeneous factors. The non-self-associate case requires a modification
of the construction of Michler--Olsson in order to allow nonlinear
elementary representations. In the self-associate case we make the
projective extension and its factor set explicit. 
We then assemble the homogeneous factors by means of an intermediate
Young subgroup and Clifford theory, obtaining a parametrization of the
irreducible spin representations of $H_\beta^+$ by the admissible
$\beta$-tuples $\mathcal C_\beta\in\mathfrak C_\beta$ introduced in
Notation~\ref{not:admissible-beta-tuples}; see
Theorem~\ref{thm:param-level}.
To study the Galois action, we introduce an index-two intertwining
criterion (see Lemma~\ref{lem:galois-index-two}) and then prove a
propagation lemma describing its behavior under Humphreys products (see
Lemma~\ref{lem:BN-propagation}).
For the unique self-associate elementary factor, an explicit intertwining
construction shows that the Galois contribution of a homogeneous component
labelled by $\lambda$ is $ \beta\,\ell(\lambda)\mod2 $.

On the other hand, the homogeneous factors attached to non-self-associate
elementary representations contribute trivially. It follows that, at level
$\beta$, the Galois action depends only on the zero component of
$\mathcal C_\beta$, which we denote by $\lambda_\beta^{(0)}$, and is
controlled by
\[
\beta\,\ell\bigl(\lambda_\beta^{(0)}\bigr)\mod2.
\]
Then we identify the local parameters $\mathcal C_\beta$ with the
successive rows of the $p$-bar core tower and assemble the different
$p$-adic levels. The resulting local Galois exponent coincides with the
global one computed in Section~\ref{sec:part1}. This yields the
$G(p)$-equivariant bijection of Theorem~\ref{thm:main}, and the
corresponding result for the double cover of the alternating group follows
by Clifford theory.

Finally, in Section~6, we prove a blockwise version of
Theorem~\ref{thm:main}, replacing $p'$-characters by height-zero
characters in a fixed spin $p$-block and in its Brauer correspondent.

\section{Spin characters and the global Galois action }
\label{sec:part1}

Throughout this section, we refer to \cite{olsson} for further details
on spin characters, bar partitions, and their combinatorics.
Let $n$ be a positive integer. We consider the double covering group
$\tSym_n$ of the symmetric group $\sym_n$ defined by the presentation
$$
\tSym_n=\langle t_1,\ldots,t_{n-1},z\mid z^2=1,\,t_j^2=z,
(t_jt_{j+1})^3=z,(t_jt_k)^2=z\ (|j-k|\geq 2)
\rangle.
$$
The sign character of $\sym_n$ lifts through the natural projection
$\pi:\tSym_n\longrightarrow \sym_n$ to a linear character $
\varepsilon:\tSym_n\longrightarrow\{\pm1\} $, whose kernel is the
double cover $ \tAlt_n=\ker(\varepsilon) $ of $\alt_n$. Thus,
$ \tAlt_n\triangleleft\tSym_n$ and 
$[\tSym_n:\tAlt_n]=2$. 
In the following, if $K$ is a subgroup of $\sym_n$, then we set
\begin{equation}
\label{eq:Ktilde}
K^+=\pi^{-1}(K).
\end{equation}
we write, for $1\leq j<n$
\[
s_j=\pi(t_j)=(j\ j+1).
\]
Independently, for each partition $\mu$ of $n$, we choose an element
$s_\mu\in\sym_n$ of cycle type $\mu$ and a lift $t_\mu\in\tSym_n$ with
$\pi(t_\mu)=s_\mu$.
Such an element is said to have \emph{cycle type}
$\mu$. Note that the two elements $t_{\mu}$ and $zt_{\mu}$
have the same cycle type. 
Furthermore, the conjugacy class of \(t_{\mu}\) in \(\tSym_n\) is
called a \emph{non-split class} if \(t_{\mu}\) and
\(zt_{\mu}\) are conjugate in \(\tSym_n\), and a \emph{split class}
otherwise. Similarly, whenever \(t_{\mu}\in\tAlt_n\), its
conjugacy class in \(\tAlt_n\) is called \emph{non-split} if
\(t_{\mu}\) and \(zt_{\mu}\) are conjugate in \(\tAlt_n\), and
\emph{split} otherwise.

We denote by $\mathcal D_n$ the set of bar partitions of $n$, and by 
$\mathcal{O}_n$ the set of partitions of $n$ with
odd parts. Schur proved (see~\cite{schur}, \S7 and p.\,176)
that the $\tSym_n$-split classes are labeled by $\mathcal{O}_n\cup
\mathcal{D}^-_n$, whereas the $\tAlt_n$-split classes are labeled by
$\mathcal O_n \cup \mathcal D_n^+$. 

\subsection{Association and index-two Clifford theory}
\label{subsec:indice2}

Let $G$ be a finite group and let $H\triangleleft G$ be a subgroup of
index two. We denote by $\varepsilon:G\longrightarrow\{\pm1\}$
the non-trivial linear character with kernel $H$.
Multiplication $ \chi\longmapsto\varepsilon\chi$ defines an action on
$\Irr(G)$.
Assume the orbits of this action are parametrized by a set
$\Lambda$. There are two cases.
\begin{enumerate}[$\star$]
\item \textbf{(SA) Self-associate case.} If the orbit labelled by
$\lambda$ is a singleton, we denote its unique element by
$\chi_\lambda\in\Irr(G)$. Thus, $\varepsilon\chi_\lambda=\chi_\lambda$. 
Let \(\rho_\lambda\) be a representation affording \(\chi_\lambda\).
An \emph{associator} of \(\rho_\lambda\) is an invertible operator \(J\)
such that, for all $g\in G$,
\[
J\rho_\lambda(g)J^{-1}
=
\varepsilon(g)\rho_\lambda(g).
\]
It is unique up to a non-zero scalar, and we may normalize it so that $
J^2=I$.
By Clifford theory,
$
\Res_H^G\chi_\lambda
=
\chi_\lambda^++\chi_\lambda^-,
$
where the two constituents are $H$-irreducible and distinct, and afforded
by the \(+1\)- and \(-1\)-eigenspaces of \(J\), respectively.
\item \textbf{(NSA) Non-self-associate case.}
If the orbit labelled by $\lambda$ has cardinality $2$, we denote its two
elements by $ \chi_\lambda^+$ and $\chi_\lambda^- $, where $
\chi_\lambda^-=\varepsilon\chi_\lambda^+$.
Their restrictions to $H$ coincide and are irreducible. We denote this
common restriction by $
\chi_\lambda\in\Irr(H)$. Thus, $ \Res_H^G\chi_\lambda^+ =
\Res_H^G\chi_\lambda^- = \chi_\lambda$.
\end{enumerate}
In both cases, for $\lambda\in\Lambda$, we define the character difference
\begin{equation}
\label{eq:chardiff}
\Delta_\lambda=\chi_\lambda^+-\chi_\lambda^-.
\end{equation}
We refer to the distinction between the self-associate and
non-self-associate cases as the \emph{association type}.
\subsection{Spin characters and bar partitions}

Let $n$ be a positive integer. For a bar partition $\lambda$ of $n$,
we write $|\lambda|=n$ for its size, $\ell(\lambda)$ for its number of
parts, and we set $ d(\lambda)=|\lambda|-\ell(\lambda)$.
Recall that $\mathcal D_n$ labels the $\langle \varepsilon\rangle$-orbits in
$\Irr(\tSym_n)$. 
    More precisely, let $\mathcal D_n^+$ (resp. $\mathcal D_n^-$) be the
subset of $\lambda\in\mathcal D_n$ consisting of partitions $\lambda$ such
that $d(\lambda)\equiv 0\mod 2$ (resp.\! $d(\lambda)\equiv 1\mod 2$).
Then, the elements of $\mathcal D_n^+$ label the
$\langle\varepsilon\rangle$-orbits of size $1$, while those of $\mathcal
D_n^-$ label orbits of size $2$. In particular, we are in the setting of
\S\ref{subsec:indice2}, which we now specialize to the double cover
$\tSym_n$.
To describe the irreducible spin characters of $\tSym_n$ and
$\tAlt_n$, we specialize the notation of \S\ref{subsec:indice2} by
writing $\xi_\lambda$, $\xi_\lambda^+$ and $\xi_\lambda^-$ for the
corresponding spin characters.
Since $z$ acts as $-1$ in every spin representation, every spin
character vanishes on non-split classes. Schur's character formulas
\cite[\S4.2, (18) and (19)]{BrNaMZ}
further show that, for every $\lambda\in\mathcal D_n$, the character
$\xi_\lambda$ is integer-valued.
Moreover, the difference character corresponding to $\lambda$ vanishes
on all conjugacy classes except those of cycle type $\lambda$.
More precisely, 
if $\lambda\in\mathcal D_n^-$, then $\Delta_\lambda$ is a class
function on $\tSym_n$, and
\begin{equation}
\label{eq:diffSntilte}
\Delta_\lambda(t_\lambda)=
\sqrt 2\,i^{\frac{d(\lambda)+1}{2}}
\sqrt{\lambda_1\cdots\lambda_{\ell(\lambda)}},
\end{equation}
up to a sign.
If $\lambda\in\mathcal D_n^+$, then $\Delta_\lambda$ is a class
function on $\tAlt_n$, and
\begin{equation}
\label{eq:diffAntilte}
\Delta_\lambda(t_\lambda)=
i^{\frac{d(\lambda)}{2}}
\sqrt{\lambda_1\cdots\lambda_{\ell(\lambda)}},
\end{equation}
up to a sign.

\subsection{Galois action and $p$-rational characters}

From now on, let $p$ be an odd prime dividing $|\tSym_n|$. Write $2n!=p^f
m$ with $m$ prime to $p$. By Brauer's splitting-field theorem
\cite[Theorem 10.3]{isaacs}, $K=\mathbb Q_{2n!}$ is a splitting field of
$\tSym_n$. As  recalled in the introduction, a character is $p$-rational
if and only if its values lie in $\mathbb Q_m$, or equivalently, if it is
fixed by the subgroup $G(p)$ of $\mathcal H$ acting trivially on all roots
of unity of order prime to $p$. 
Although $G(p)$ is defined independently of $K$, its action on the
characters considered here factors through its restriction to $K$. The
image of the restriction map $G(p)\longrightarrow \Gal(K/\mathbb Q)$ is
$\Gal(K/\mathbb
Q_m)\simeq(\Z/p^{f}\Z)^\times$. Since $p$ is odd, this group is cyclic.
We choose $\tau_p\in G(p)$ whose restriction to $K$ generates $\Gal(K/\mathbb
Q_m)$. 
We then have 
$$
\chi\text{ is \(p\)-rational}
\iff
\chi^{\tau_p}=\chi.
$$
For a bar partition $\lambda=(\lambda_1,\ldots,\lambda_{\ell(\lambda)})$,
we define
\begin{equation}
\label{eq:prodpart}
\pi_{\lambda}=\prod_{u=1}^{\ell(\lambda)} \lambda_u.
\end{equation}

\begin{theorem}
Let $\lambda$ be a bar partition of $n$. If $\epsilon\in\{+,-\}$, then
$$(\xi_\lambda^\epsilon)^{\tau_p}=\xi_\lambda^{\epsilon(-1)^{\nu_p(\pi_{\lambda})}},$$
where $\nu_p$ denotes the $p$-valuation and signs are multiplied in the
usual way.
\end{theorem}

\begin{proof}
For a positive integer $N$, we write $\omega_N$ for a primitive $N$-root
of unity. First, we remark that $i=\omega_4$ and $\sqrt
2=\omega_8+\omega_8^{-1}$. In particular, $\tau_p$ acts trivially on $i$
and $\sqrt 2$ since $p$ is odd.
Let $M$ be an odd integer prime to $p$ such that
$$\pi_{\lambda}=2^{\nu_2(\pi_{\lambda})}p^{\nu_p(\pi_{\lambda})}M,$$
and $M'=p^{\nu_p(\pi_{\lambda})}M$. By~(\ref{eq:diffSntilte}) and
(\ref{eq:diffAntilte}), we deduce that $\tau_p$ acts on
$\xi_\lambda^\epsilon$ as it acts on $\sqrt{M'}$.
Moreover, there is an integer $r$  prime to $M'$ such that
$\tau_p(\omega_{M'})=\omega_{M'}^r$. We set
$\omega_M=\omega_{M'}^{p^{\nu_p(\pi_\lambda)}}$. Since $M$ is prime to
$p$, $\tau_p(\omega_M)=\omega_M$. On the other hand, we have
$$\tau_p(\omega_M)=\tau_p\left(\omega_{M'}^{p^{\nu_p(\pi_\lambda)}}\right)=
\left(\omega_{M'}^{p^{\nu_p(\pi_\lambda)}}\right)^r=\omega_M^r.$$
Hence, $r\equiv 1\mod M$. Now,~\cite[Proposition 4.1]{BrNa} yields that
$\tau_p$ acts on $\sqrt{M'}$ by
$$\legendre{r}{M'}=\legendre{r}{p^{\nu_p(\pi_\lambda)}M}=
\legendre{r}{p^{\nu_p(\pi_\lambda)}}\legendre{r}{M}=
\legendre{r}{p}^{\nu_p(\pi_\lambda)}.$$
Since the restriction of $\tau_p$ generates
$\Gal(K/\mathbb Q_m)$, the residue class of $r$ modulo $p$
generates $\mathbb F_p^\times$. In particular, it is a non-square, and
$\legendre{r}{p}=-1$. 
Finally, we obtain
$$\legendre{r}{M'}=(-1)^{\nu_p(\pi_\lambda)},$$
and the result follows.
\end{proof}

\subsection{The \( p\)-bar core tower}

Each bar partition is uniquely determined by its $p$-bar core and its
$p$-bar quotient. Iterating this decomposition, we now recall how to
associate a $p$-bar core tower with every bar partition. 
Let \(\lambda\) be a bar partition. We define recursively two sequences
\((\mathcal Q_\beta(\lambda))_{\beta\geq 0}\) and \((\mathcal
C_\beta(\lambda))_{\beta\geq 0}\) of multipartitions as follows. The first
records the partitions which occur at each step of the $p$-bar-quotient
process, whereas the second records their cores and forms the $p$-bar core
tower.
We start with
$
\mathcal Q_0(\lambda)=(\lambda)$ 
and set
$\mathcal C_0(\lambda)
=
\bigl(\operatorname{core}_{\bar p}(\lambda)\bigr)$.
Let $\beta\geq 0$. Suppose that
\[
\mathcal Q_\beta(\lambda)
=
\bigl(
\alpha_\beta^{(0)},
\alpha_\beta^{(1)},
\ldots,
\alpha_\beta^{(e_\beta)}
\bigr),
\]
where \(\alpha_\beta^{(0)}\) is a bar partition and
\(\alpha_\beta^{(j)}\), for \(j>0\), are ordinary partitions.
We define the $\beta$-th level of the $p$-bar tower by
\[
\mathcal C_\beta(\lambda)
=
\bigl(
\lambda_\beta^{(0)},
\lambda_\beta^{(1)},
\ldots,
\lambda_\beta^{(e_\beta)}
\bigr),
\]
where $ \lambda_\beta^{(0)} $ is the $p$-bar core of $\alpha_\beta^{(0)}$
and for all $1\leq j\leq e_\beta$, $ \lambda_\beta^{(j)} $ is the usual
$p$-core of the partition $\alpha_\beta^{(j)}$.
To construct \(\mathcal Q_{\beta+1}(\lambda)\), we take the \(p\)-bar
quotient of the bar partition \(\alpha_\beta^{(0)}\), and the ordinary
\(p\)-quotient of each \(\alpha_\beta^{(j)}\), \(j>0\). The \(p\)-bar
quotient of \(\alpha_\beta^{(0)}\) consists of one bar partition and
\((p-1)/2\) ordinary partitions, while the \(p\)-quotient of each ordinary
partition \(\alpha_\beta^{(j)}\) consists of \(p\) ordinary partitions.
Listing all these components, with the unique bar component first, defines
$\mathcal Q_{\beta+1}(\lambda)$.
Hence, the sequence $(e_\beta)_{\beta\geq 0}$ satisfies $
e_{\beta+1}=\frac{1}{2}(p-1)+pe_\beta$. It follows by induction that 
\begin{equation}
\label{eq:defei}
e_\beta=\frac{1}{2}(p^\beta-1).
\end{equation}

\begin{theorem} 
\label{thm:globalprat}
Let $\lambda$ be a bar partition. With the notation 
above, we have
\[
\nu_p\!\left(\pi_\lambda\right)
\equiv
\sum_{\beta\geq1}\beta \,\ell(\lambda_\beta^{(0)})
\mod2.
\]
\end{theorem}

\begin{proof}
The parts of the zero component of the \(p\)-bar quotient consist of the
integers $a/p$, where $a$ is a part of $\lambda$ divisible by $p$, Hence, 
the parts of \(\alpha_\beta^{(0)}\) are obtained from the parts of \(\lambda\)
divisible by \(p^\beta\), after division by \(p^\beta\), and it follows that
\[
\ell(\alpha_\beta^{(0)})
=
|\{u\mid p^\beta\mid\lambda_u\}|,
\]
and therefore
\begin{equation}
\label{eq:preuvgeint1}
\nu_p\!(\pi_\lambda)=\nu_p\!\left(\prod_u\lambda_u\right)
=
\sum_u\nu_p(\lambda_u)
=
\sum_{k\geq1}\ell(\alpha_k^{(0)}).
\end{equation}
By \cite[Lemma 3.11]{BrNaMZ}, for every $k\geq0$,
\[
\ell(\alpha_k^{(0)})
\equiv
\ell(\lambda_k^{(0)})
+
\ell\bigl((\alpha_k^{(0)})^{[\bar p]}\bigr)
\mod2,
\]
where $(\alpha_k^{(0)})^{[\bar p]}$ denotes the $p$-bar cocore of
$\alpha_k^{(0)}$. The non-zero components of its $p$-bar quotient
contribute an even number of parts, and its zero component is
$\alpha_{k+1}^{(0)}$. Hence,
\[
\ell(\alpha_k^{(0)})
\equiv
\ell(\lambda_k^{(0)})
+
\ell(\alpha_{k+1}^{(0)})
\mod2.
\]
Iterating this relation gives, for every \(k\geq0\),
\begin{equation}
\label{eq:length-parity-tower}
\ell(\alpha_k^{(0)})
\equiv
\sum_{\beta\geq k}\ell(\lambda_\beta^{(0)})
\mod2.
\end{equation}
We then deduce from \eqref{eq:preuvgeint1} that
\[
\nu_p\!(\pi_\lambda)
\equiv
\sum_{k\geq1}\sum_{\beta\geq k}\ell(\lambda_\beta^{(0)})
=
\sum_{\beta \geq1} \beta\,\ell(\lambda_\beta^{(0)})
\mod2.
\]
This proves the claim.
\end{proof}

\section{Local spin \(p'\)-characters}
\label{sec:part2}

\subsection{The reduced local normalizer}

Write the $p$-adic expansion of $n$ as
\begin{equation}
\label{eq:expansionadicn}
n=\sum_{\beta\geq0}n_\beta p^\beta,
\end{equation}
with $0\leq n_\beta<p$, and let $P\in\Syl_p(\tSym_n)$ and $
N=N_{\tSym_n}(P)$. We are interested in the spin characters of $ N$ of
degree prime to $p$. Following the reduction used by Michler and Olsson
\cite{MichlerOlsson}, we may first consider the quotient of
$$\overline N=N/P'.$$ 
Indeed, let $\chi\in\Irr_{p'}(N)$ and let $\theta\in\Irr(P)$ lie below
$\chi$. By Clifford's theorem (see for example \cite[Theorem
(6.2)]{isaacs}), there is a positive integer $e$ such that
\[
\chi(1)
=
e\,[N:I_{N}(\theta)]\,\theta(1),
\]
where $I_{N}(\theta)$ denotes the inertia subgroup of $\theta$. Since
$\theta(1)$ is a power of $p$, we must have $\theta(1)=1$, and 
\[
P'\leq\ker\chi.
\]
Conversely, $P/P'$ is an abelian normal Sylow $p$-subgroup of $\overline
N$. Hence, by It\^o's theorem~\cite[Theorem (6.15)]{isaacs}, every
irreducible character of $\overline N$ has degree prime to $p$. Then
inflation induces a bijection
\[
\Irr^{\mathrm{spin}}(\overline N)
\longrightarrow
\Irr_{p'}^{\mathrm{spin}}(N).
\]
For $\beta\geq1$, let $X_\beta\in\Syl_p(\sym_{p^\beta})$. Let
$L_\beta=N_{\sym_{p^\beta}}(X_\beta)$ and $\mathcal H_\beta= L_\beta\wr
\sym_{n_\beta}\leq \sym_{p^\beta n_\beta}$. We have $(X_\beta')^{n_\beta}
\triangleleft \mathcal H_\beta$.
Since $p$ is odd, $X_\beta^+$ has a unique Sylow $p$-subgroup, which maps
isomorphically onto $X_\beta$. By abuse of notation, we denote it again by
$X_\beta$. In particular, $(X_\beta')^{n_\beta} \triangleleft \mathcal
H_\beta^+$. We define
\[
H_\beta=\mathcal H_\beta/(X_\beta')^{n_\beta}
\quad\text{and}\quad
H_\beta^+=\mathcal H_\beta^+/(X_\beta')^{n_\beta}.
\]
Now, if $N_\beta=L_\beta/X_\beta'$, then 
$$H_\beta\simeq N_\beta\wr \sym_{n_\beta}\quad \text{and}\quad
H_\beta^+\simeq (N_\beta\wr
\sym_{n_\beta})^+.$$
The description in \cite[\S3, pp.~87--92]{MichlerOlsson} yields
\begin{equation}
\label{eq:reduced-local-decomposition}
\overline N
\simeq
\tSym_{n_0}
\widehat\times
\widehat{\prod}_{\beta\geq1} H_\beta^+.
\end{equation}

\subsection{Spin characters of \(N_\beta^+\)}

Let $\beta\geq 1$. We now determine the irreducible spin characters of
\[
N_\beta^+=L_\beta^+/X_\beta'.
\]
The structure of $N_\beta=L_\beta/X_\beta'$ is particularly simple. Let
\[
G_p=\mathbb F_p\rtimes\mathbb F_p^\times,
\]
where $\mathbb F_p$ is viewed as an additive group and $\mathbb
F_p^\times$ acts by multiplication.
The description of $L_\beta$ in \cite{MichlerOlsson} gives $ N_\beta\simeq
G_p^\beta. $ Moreover, $N_\beta^+$ admits a linear spin character
$\Lambda_\beta$. Hence tensoring by $\Lambda_\beta$ identifies the
irreducible spin characters of $N_\beta^+$ with the irreducible characters
of $G_p^\beta$. 

Recall that $G_p$ has $p-1$ linear characters and a unique nonlinear
irreducible character $\Phi$, of degree $p-1$. For
$\boldsymbol\phi=(\phi_1,\ldots,\phi_\beta)\in \Irr(G_p)^\beta$, we write
$$\phi_1\boxtimes\cdots\boxtimes\phi_\beta \in \Irr(N_\beta),$$
where $\boxtimes$ denotes the external tensor product. Thus, the
irreducible spin characters of $N_\beta^+$ are
\[
\rho_{\boldsymbol\phi}
=
\Lambda_\beta\otimes
\Inf\bigl(\phi_1\boxtimes\cdots\boxtimes\phi_\beta\bigr),
\]
where $\Inf$ is the lift from $N_\beta$ to $N_{\beta}^+$.

Let $\nu$ denote the unique non-trivial linear character of order $2$ of
$G_p/\mathbb F_p\simeq \mathbb F_p^\times$. Under the parametrization
above, association is given by
$
\rho_{\boldsymbol\phi}
\longmapsto
\rho_{\boldsymbol\phi}\otimes\varepsilon
=
\rho_{(\nu\phi_1,\ldots,\nu\phi_\beta)}$.
Since $ \nu\phi=\phi $ if and only if $\phi=\Phi$,
it follows that, up to equivalence, $N_\beta^+$ has a unique
self-associate irreducible spin representation $R_\beta$, whose character
is 
\[
\chi_{R_\beta}
=
\Lambda_\beta\otimes\Inf\bigl(\Phi^{\boxtimes \beta}\bigr).
\]
All the remaining irreducible spin characters are non-self-associate.
Since $G_p$ has $p$ irreducible characters, $N_\beta^+$ has $p^\beta$
irreducible spin characters, and hence there are
\[
e_\beta=\frac{p^\beta-1}{2}
\]
non-self-associate association classes.

\begin{remark}
The preceding description differs from the linearity assertion used in
the corresponding local construction of Michler and Olsson
\cite{MichlerOlsson}. Indeed, if exactly $k$ of the components
$\phi_j$ of $\boldsymbol\phi$ are equal to $\Phi$, then
\[
\rho_{\boldsymbol\phi}(1)=(p-1)^k.
\]
Thus, for $\beta\geq2$, non-self-associate elementary spin characters of
degree greater than $1$ occur whenever $1\leq k<\beta$. 
\end{remark}

\subsection{Homogeneous factors}

Let $\beta\geq 1$, and $r\geq 1$. We call a representation of
\[
(N_\beta^+)^{\widehat\times r}
\]
\emph{homogeneous} if all its elementary factors belong to the same
association class of spin representations of $N_\beta^+$. The
corresponding subgroup
\[
(N_\beta\wr \sym_r)^+
\]
will be called \emph{the associated homogeneous factor}.
We now describe the irreducible spin representations lying above such a
homogeneous representation. We follow the construction of Michler and
Olsson as closely as possible. However, in the non-self-associate case, a
modification is required when the elementary character is nonlinear.
Furthermore, in the self-associate case, we make the projective structure
explicit by determining the factor set associated with a chosen
transversal.

We first recall the realization of the base subgroup introduced by
Michler and Olsson. If $t_1,\ldots,t_{r-1}$ denote the chosen lifts of
the simple transpositions, then, for $x\in N_\beta^+$, we define
\[
x^{(1)}=x,
\qquad
x^{(j+1)}=t_jx^{(j)}t_j^{-1}
\quad (\text{for }1\leq j<r).
\]
We shall also write
\[
S_{\beta,r}^+
=
\langle t_1,\ldots,t_{r-1},z\rangle.
\]
By \cite[Lemma~3.5]{MichlerOlsson}, the covering group induced on the
permutations of the $r$ blocks depends on $p^\beta$ modulo $4$. More
precisely,
\begin{equation}
\label{eq:splus1}
t_j^2=
\begin{cases}
z,&p^\beta\equiv1\mod4,\\
1,&p^\beta\equiv-1\mod4,
\end{cases}
\quad\text{and}\quad
(t_jt_{j+1})^3
=
\begin{cases}
z,&p^\beta\equiv1\mod4,\\
1,&p^\beta\equiv-1\mod4,
\end{cases}
\end{equation}
and
\begin{equation}
\label{eq:splus2}
(t_jt_k)^2=z
\qquad (|j-k|>1)
\end{equation}
in both cases. Following \cite[Lemma~4.2(8)]{MichlerOlsson},
we set
\begin{equation}
\label{eq:defsignechoix}
\gamma_\beta=
\begin{cases}
1,&p^\beta\equiv1\mod4,\\
\mathrm i,&p^\beta\equiv-1\mod4.
\end{cases}
\end{equation}
Thus, the matrix corresponding to $t_j$ involves $\gamma_\beta D_j$.
For $1\leq j\leq r$, the elements $x^{(j)}$, with $x\in N_\beta^+$, form
the $j$-th copy of $N_\beta^+$ in the base subgroup. These copies generate
\[
(N_\beta^+)^{\widehat\times r},
\]
and satisfy the following relations.
\begin{enumerate}[(i)]
\item For each fixed $j$, the elements $x^{(j)}$ satisfy the
defining relations of $N_\beta^+$.

\item If $j\neq k$, then
\[
x^{(j)}y^{(k)}
=
z^{\partial(x)\partial(y)}
y^{(k)}x^{(j)},
\]
where $\partial(x)\in\{0,1\}$ is determined by
$
\varepsilon(x)=(-1)^{\partial(x)}.
$

\item The central involutions of the different copies are
identified with the same element $z$.
\end{enumerate}
In particular, every element of the base can be written in the form
\[
x_1^{(1)}x_2^{(2)}\cdots x_r^{(r)},
\qquad x_1,\ldots,x_r\in N_\beta^+.
\]
Moreover, \cite[p.~92]{MichlerOlsson} gives
\[
t_jx^{(k)}t_j^{-1}
=
\begin{cases}
x^{(j+1)},&k=j,\\
x^{(j)},&k=j+1,\\
z^{\partial(x)}x^{(k)},&k\neq j,j+1.
\end{cases}
\]

\subsubsection{Non-self-associate factors.}

Let $\rho$ be an irreducible non-self-associate spin representation of
$N_\beta^+$, acting on a space $V$. We consider the homogeneous base
representation determined by $r$ copies of $\rho$. Then the construction
of Michler and Olsson extends to $\rho$, without any linearity assumption.

Let $W_r$ be the space on which the matrices $ F_1,\ldots,F_r $ and $
D_1,\ldots,D_{r-1} $ (of Michler--Olsson's construction
\cite[Lemma~4.2]{MichlerOlsson}) act. We choose the ordering so that the
$j$-th copy corresponds to $F_{r-j+1}$.

\begin{notation}
Let $V$ be a vector space and let $r\geq 1$. For every $M\in\End(V)$ and
every $1\leq j\leq r$, we write
\[
M^{(j)}
=
I^{\otimes(j-1)}\otimes M\otimes I^{\otimes(r-j)}
\in \End(V^{\otimes r}).
\]
\end{notation}

For $x\in N_\beta^+$ and $1\leq j\leq r$, we define
\[
E_\rho(x^{(j)})
=
\rho(x)^{(j)}\otimes F_{r-j+1}^{\partial(x)}
\in
\GL(V^{\otimes r}\otimes W_r).
\]

\begin{proposition}
With the notation above, these formulas define a representation
\[
E_\rho:
(N_\beta^+)^{\widehat\times r}
\longrightarrow
\GL(V^{\otimes r}\otimes W_r).
\]
Moreover, $E_\rho$ extends to a representation $\widetilde E_\rho$ of
the associated homogeneous factor $(N_\beta\wr \sym_r)^+$ by setting
\begin{equation}
\label{eq:defopnsa}
\widetilde E_\rho(t_j)=P_j\otimes \gamma_\beta D_j,
\end{equation}
where $P_j$ interchanges the $j$-th and $(j+1)$-st factors of $V^{\otimes
r}$, and $\gamma_\beta$ is defined in \eqref{eq:defsignechoix}.
\end{proposition}

\begin{proof}
To prove that the formula defines a representation, it is enough to check that
the prescribed operators $E(x^{(j)})$ satisfy relations~(i), (ii) and
(iii) above. Although the verification is routine, we include a few
details for the convenience of the reader. 
We first consider relation~(i). For $1\leq j\leq r$, 
\[
x\longmapsto
\rho(x)^{(j)}\otimes F_{r-j+1}^{\partial(x)}
\]
is a representation of $N_\beta^+$. Indeed, we have
$
\partial(xy)\equiv\partial(x)+\partial(y)\mod 2
$
and $F_{r-j+1}^2=I$, thus
\[
\begin{aligned}
E_\rho(x^{(j)})E_\rho(y^{(j)})
&=
\rho(xy)^{(j)}\otimes
F_{r-j+1}^{\partial(x)+\partial(y)}\\
&=
\rho(xy)^{(j)}\otimes
F_{r-j+1}^{\partial(xy)}\\
&=
E_\rho((xy)^{(j)}).
\end{aligned}
\]
We next verify relation~(ii). If $j\neq k$, then $\rho(x)^{(j)}$ and
$\rho(y)^{(k)}$ act on distinct tensor factors and therefore commute.
On the other hand, since
\[
F_uF_v=-F_vF_u
\quad \text{for }u\neq v,
\]
we have
\[
F_{r-j+1}^{\partial(x)}F_{r-k+1}^{\partial(y)}
=
(-1)^{\partial(x)\partial(y)}
F_{r-k+1}^{\partial(y)}F_{r-j+1}^{\partial(x)}.
\]
Hence,
\[
\begin{aligned}
E_\rho(x^{(j)})E_\rho(y^{(k)})
&=
(-1)^{\partial(x)\partial(y)}
E_\rho(y^{(k)})E_\rho(x^{(j)})\\
&=
E_\rho\!\left(
z^{\partial(x)\partial(y)}y^{(k)}x^{(j)}
\right),
\end{aligned}
\]
since $\rho(z)=-I$.
Finally, relation~(iii) is satisfied since $\partial(z)=0$ and
$\rho(z)=-I$, so that
\[
E_\rho(z^{(j)})=\rho(z)^{(j)}=-I
\]
for every $j$.
We now extend $E_\rho$ to the associated homogeneous factor. Since
$(N_\beta\wr \sym_r)^+$ is generated by the base subgroup and by
$t_1,\ldots,t_{r-1}$, it remains to check that the operators defined
in~\eqref{eq:defopnsa} satisfy the defining relations involving the
generators $t_j$: first their conjugation action on the base, and then the
relations among the $t_j$ themselves. 
By~\cite[Lemma 4.2(5)]{MichlerOlsson}, we have
\begin{align*}
  D_j F_{r-j+1} D_j^{-1} &= F_{r-j},\quad 
  D_j F_{r-j} D_j^{-1}   = F_{r-j+1},\\
  D_j F_k D_j^{-1} &= -F_k \quad \text{for } k \neq r-j, r-j+1.
\end{align*}
Hence,
\[
\widetilde{E}_\rho(t_j)E_\rho(x^{(j)})\widetilde{E}_\rho(t_j)^{-1}
=
E_\rho(x^{(j+1)})
\quad\text{and}\quad
\widetilde{E}_\rho(t_j)E_\rho(x^{(j+1)})\widetilde{E}_\rho(t_j)^{-1}
=
E_\rho(x^{(j)}).
\]
If $k\neq j,j+1$, then
\[
\begin{aligned}
\widetilde{E}_\rho(t_j)E_\rho(x^{(k)})\widetilde{E}_\rho(t_j)^{-1}
&=
\rho^{(k)}(x)\otimes
D_jF_{r-k+1}^{\partial(x)}D_j^{-1}\\
&=
(-1)^{\partial(x)}E_\rho(x^{(k)})\\
&=
E_\rho\bigl(z^{\partial(x)}x^{(k)}\bigr).
\end{aligned}
\]
Furthermore, we have
\begin{equation}
\label{eq:relationfliptresse}
P_j^2=I,\quad
P_jP_{j+1}P_j=P_{j+1}P_jP_{j+1}
\quad\text{and}\quad
P_jP_k=P_kP_j\quad(|j-k|>1).
\end{equation}
On the other hand, \cite[Lemma~4.2(1), (3), (6)]{MichlerOlsson} gives
\[
D_j^2=-I,\quad
(D_jD_{j+1})^3=-I
\quad\text{and}\quad
(D_jD_k)^2=-I\quad(|j-k|>1).
\]
It follows that
\[
\widetilde E_\rho(t_j)^2
=
P_j^2\otimes \gamma_\beta^2D_j^2
=
\begin{cases}
-I=E_\rho(z),&p^\beta\equiv1\mod4,\\
 I,&p^\beta\equiv-1\mod4.
\end{cases}
\]
Moreover, using that $ (P_jP_{j+1})^3=I$ and $(D_jD_{j+1})^3=-I $, we
obtain
\[
\begin{aligned}
\bigl(
\widetilde E_\rho(t_j)
\widetilde E_\rho(t_{j+1})
\bigr)^3
&=
(P_jP_{j+1})^3
\otimes
\gamma_\beta^6(D_jD_{j+1})^3\\
&=
\begin{cases}
-I=E_\rho(z),&p^\beta\equiv1\mod4,\\
 I,&p^\beta\equiv-1\mod4.
\end{cases}
\end{aligned}
\]
Finally, if $|j-k|>1$, then
$
(P_jP_k)^2=I$ and
$(D_jD_k)^2=-I$,
and $\gamma_\beta^4=1$. Hence,
\[
\begin{aligned}
\bigl(
\widetilde E_\rho(t_j)
\widetilde E_\rho(t_k)
\bigr)^2
&=
(P_jP_k)^2
\otimes
\gamma_\beta^4(D_jD_k)^2\\
&=
-I
=
E_\rho(z).
\end{aligned}
\]
Thus relations~\eqref{eq:splus1} and~\eqref{eq:splus2} are satisfied,
and the result follows.
\end{proof}

For every partition $\lambda\vdash r$, Gallagher's
theorem~\cite[Corollary~(6.17)]{isaacs} then gives an irreducible
representation
\begin{equation}
\label{eq:basensahom}
E_{\rho,\lambda}
=
\widetilde{E}_\rho\otimes\Inf S^\lambda,
\end{equation}
where $S^\lambda$ is an irreducible representation of $\sym_r$
corresponding to $\lambda$.

\begin{proposition}
\label{prop:asso1}
The representation $E_{\rho,\lambda}$ is self-associate if and only if $r$
is even. Moreover, if $r$ is even, an associator is given by
\[
J_{\rho,r,\lambda}
=
I_{V^{\otimes r}}\otimes c_r\Omega_r\otimes I_{S^\lambda},
\qquad
\Omega_r=F_rF_{r-1}\cdots F_1,
\]
where $c_r$ is chosen so that $(c_r\Omega_r)^2=I$.
\end{proposition}

\begin{proof}
Suppose that $r$ is even. 
Since $\Omega_r$ anticommutes with every $F_k$, we have
\[
\Omega_rD_j\Omega_r^{-1}
=
-D_j.
\]
Choose a scalar $c_r$ so that
\[
(c_r\Omega_r)^2=I,
\]
and set
\[
J_{\rho,r}
=
I_{V^{\otimes r}}\otimes c_r\Omega_r.
\]
Then
\[
J_{\rho,r}E_\rho(x^{(j)})J_{\rho,r}^{-1}
=
\varepsilon(x)E_\rho(x^{(j)})
\quad\text{and}\quad
J_{\rho,r}\widetilde{E}_\rho(t_k)J_{\rho,r}^{-1}
=
-\widetilde{E}_\rho(t_k).
\]
Thus, $J_{\rho,r}$ is an associator for $\widetilde{E}_\rho$, and the
result follows.
\end{proof}

\subsubsection{The self-associate factor.}

We now turn to the unique self-associate spin representation $R_\beta$ of
$N_\beta^+$.

\begin{lemma}
With the notation above, $R_\beta$ admits an associator $J_\beta$ such that
\begin{equation}
\label{eq:assoD}
J_\beta^2=I
\quad\text{and}\quad
J_\beta R_\beta(x)J_\beta^{-1}
=
\varepsilon(x)R_\beta(x)
\quad \text{for }x\in N_\beta^+.
\end{equation}
\end{lemma}

\begin{proof}
Let $a$ be a generator of $\mathbb F_p^\times$. Set
\[
b=(1,1)\in G_p
\quad\text{and}\quad
c=(0,a)\in G_p.
\]
Then $b$ generates the normal subgroup of order $p$ of $G_p$ and
$cbc^{-1}=b^a$. 
Let $\beta$ be the non-trivial linear character of $\langle b\rangle$
defined by
$\beta(b)=\zeta_p$.
The unique nonlinear irreducible character $\Phi$ of $G_p$ is afforded
by
\[
\varphi=\Ind_{\langle b\rangle}^{G_p}\beta.
\]
Using the induced model with basis
$(e_u)_{u\in\mathbb F_p^\times}$, we define
\[
\varphi(b)e_u=\zeta_p^u e_u
\quad\text{and}\quad
\varphi(c)e_u=e_{a^{-1}u}.
\]
Let $\eta$ be the unique character of order $2$ of
$\mathbb F_p^\times$, and denote again by $\eta$ its inflation to
$G_p$. Define, for $u\in\mathbb F_p^{\times}$
\begin{equation}
\label{eq:defassbase}
Je_u=\eta(u)e_u.
\end{equation}
Then $J^2=I$. Moreover, $J$ commutes with $\varphi(b)$, and
$J\varphi(c)J^{-1} = \eta(a)\varphi(c)$.
Since $b$ and $c$ generate $G_p$, it follows that, for $g\in G_p$,
\[
J\varphi(g)J^{-1}
=
\eta(g)\varphi(g).
\]
Let $ \varphi_\beta=\varphi^{\boxtimes \beta} $ be the corresponding
representation of $ N_\beta $ and set
\[
J_\beta=J^{\otimes \beta}.
\]
If $g=(g_1,\ldots,g_\beta)\in G_p^\beta$, then
\[
J_\beta\varphi_\beta(g)J_\beta^{-1}
=
\eta(g_1)\cdots\eta(g_\beta)\varphi_\beta(g).
\]
Under the identification $N_\beta\simeq G_p^\beta$, the restriction of
$\varepsilon$ to each factor $G_p$ is $\eta$, and 
\[
\varepsilon|_{N_\beta}=\eta^{\boxtimes \beta}.
\]
Hence,
\[
J_\beta\varphi_\beta(g)J_\beta^{-1}
=
\varepsilon(g)\varphi_\beta(g).
\]
If $x\in N_\beta^+$ and $\bar x$ denotes its image in
$N_\beta\simeq G_p^\beta$, then
\[
R_\beta(x)=\Lambda_\beta(x)\varphi_\beta(\bar x).
\]
Since $\Lambda_\beta(x)$ is a scalar, it commutes with $J_\beta$, and thus
\[
J_\beta R_\beta(x)J_\beta^{-1}
=
\varepsilon(x)R_\beta(x).
\]
Finally, $J_\beta^2=I$, since $J^2=I$.
\end{proof}

\begin{lemma}
\label{lem:TD}
Let $V$ be the representation space of $R_\beta$, and let
$J_\beta$ be an associator for $R_\beta$.
Set
\begin{equation}
\label{eq:defQ}
Q
=
\frac12\bigl(
I\otimes I
+J_\beta\otimes I
+I\otimes J_\beta
-J_\beta\otimes J_\beta
\bigr)
\in\End(V\otimes V).
\end{equation}
Then $Q^2=I$, and for every $x\in N_\beta^+$,
\[
Q(R_\beta(x)\otimes I)Q^{-1}
=
R_\beta(x)\otimes J_\beta^{\partial(x)}
\quad \text{and}\quad
Q\bigr(J_\beta^{\partial(x)}\otimes R_\beta(x)\bigr)Q^{-1}
=
I\otimes R_\beta(x).
\]
Moreover, if we set $P: V\otimes V\longrightarrow V\otimes V,\ v\otimes
w\longmapsto w\otimes v$, then
$
PQ=QP.
$
\end{lemma}

\begin{proof}
For $x\in N_\beta^+$, we define $\pi(x)=R_\beta(x)\otimes I$ and
$\pi'(x)=R_\beta(x)\otimes J_\beta^{\partial(x)}$. For
$\epsilon\in\{\pm\}$, we set $V^{\epsilon}=\ker(J_\beta-\epsilon I)$ the
eigenspace $J_\beta$ associated with $\epsilon$. We remark that
$\pi(x)=\pi'(x)$ on $V\otimes V^+$ and $\pi'(x)=J_\beta R_\beta(x)
J_{\beta}^{-1}\otimes I$ on $V$. Thus, 
$$Q=\left\{\begin{array}{c}
I\otimes I \ \text{on } V\otimes V^+\\
J_\beta\otimes I\ \text{on } V\otimes V^-
\end{array}
\right.$$
intertwines $\pi(x)$ and $\pi'(x)$. In particular, if we denote by
$P^{\epsilon}$ the projections onto $V^\epsilon$, then we have
$$Q=I\otimes P^++J_\beta\otimes P^-.$$
Now, using the expression
\[
P^\epsilon=\frac12(I+\epsilon J_\beta),
\]
we deduce the expression \eqref{eq:defQ}.
By a symmetric argument, the second relation is intertwined by 
$P^+\otimes I + P^-\otimes J_\beta$, which is equal to $Q$.
Finally,
\[
Q^2
=
I\otimes P^+ + J_\beta^2\otimes P^-
=
I\otimes (P^++P^-)= I\otimes I.
\]
The defining expression of $Q$ is invariant under exchange of the two
tensor factors. Hence, $PQ=QP$.
\end{proof}

Let $r\geq 1$.
Define $E:(N_\beta^+)^{\widehat{\times} r}\longrightarrow \GL(V^{\otimes
r})$ by setting, for all $x\in N_\beta^+$,
\begin{equation}
\label{eq:defESA}
E\bigr(x^{(j)}\bigr)
=
\bigr(J_\beta^{(1)}
\cdots
J_\beta^{(j-1)}\bigr)^{\partial(x)}
R_\beta(x)^{(j)}. 
\end{equation}

\begin{remark}
The formulas in~\eqref{eq:defESA} indeed define a representation. It is enough to
check that the prescribed operators $E(x^{(j)})$ satisfy
relations~{(i)}--{(iii)} above.
For relation~{(i)}, let $x,y\in N_\beta^+$ and fix $j$. Since
$R_\beta(x)^{(j)}$ commutes with $J^{(1)},\ldots,J^{(j-1)}$, we have
\[
\begin{aligned}
E\bigr(x^{(j)}\bigr)E\bigr(y^{(j)}\bigr)
&=
\bigr(J_\beta^{(1)}\cdots J_\beta^{(j-1)}\bigr)^{\partial(x)+\partial(y)}
R_\beta(xy)^{(j)}\\
&=
E\bigr((xy)^{(j)}\bigr),
\end{aligned}
\]
because
$
\partial(xy)\equiv\partial(x)+\partial(y)\mod2
$
and $J_k^2=I$.
We next verify relation~{(ii)}. Assume that $j<k$.  Since
\[
R_\beta(x)^{(j)}(J_\beta^{(j)})^{\partial(y)}
=
(-1)^{\partial(x)\partial(y)}
(J_\beta^{(j)})^{\partial(y)}R_\beta(x)^{(j)},
\]
we obtain
\[
E\bigr(x^{(j)}\bigr)E\bigr(y^{(k)}\bigr)
=
(-1)^{\partial(x)\partial(y)}
E\bigr(y^{(k)}\bigr)E\bigr(x^{(j)}\bigr) 
=
E(z)^{\partial(x)\partial(y)}
E\bigr(y^{(k)}\bigr)E\bigr(x^{(j)}\bigr).
\]
Finally, relation~{(iii)} is satisfied since, for every $j$, $
E\bigr(z^{(j)}\bigr)
= R_\beta(z)^{(j)} = -I $.
\end{remark}

For $1\leq j<r$, let $Q_j$ and $P_j$ denote respectively the operators
$Q$ and $P$ acting on the $j$th and $(j+1)$st tensor factors and as the
identity on all the others. In particular,~\eqref{eq:defQ} gives
\[
Q_j
=
\frac12\Bigl(
I+J_\beta^{(j)}+J_\beta^{(j+1)}
-J_\beta^{(j)}J_\beta^{(j+1)}
\Bigr).
\]
We define 
$$
U_j=P_jQ_j.$$

\begin{lemma}
\label{lem:coxUi}
The operators $U_1,\ldots,U_{r-1}$ satisfy the Coxeter relations
\[
U_j^2=I,\quad
U_jU_k=U_kU_j\quad\text{if }|j-k|>1
\quad\text{and}\quad
U_jU_{j+1}U_j
=
U_{j+1}U_jU_{j+1}.
\]
Hence, $s_j\mapsto U_j$ defines a representation of
$\sym_r$ on $V^{\otimes r}$.
\end{lemma}

\begin{proof}
For distinct $a,b\in\{1,\ldots,r\}$, we define
$$Q_{a,b}=\frac{1}{2}(I+J_\beta^{(a)} + J_\beta^{(b)} -
J_\beta^{(a)}J_\beta^{(b)}).$$
In particular, $Q_j=Q_{j,j+1}$. All the operators $Q_{a,b}$ commute with one
another, and 
\[
P_jQ_{a,b}P_j^{-1}
=
Q_{s_j(a),s_j(b)}.
\]
Since $P_j$ commutes with $Q_j$ and both are involutions, we have
$U_j^2=I$.
If $|j-k|>1$, then $P_j$ commutes with $Q_k$ and $P_k$ commutes with
$Q_j$, so that $ U_jU_k=U_kU_j$.
Finally, we have
\[
U_jU_{j+1}U_j
=
P_jP_{j+1}P_j\,
Q_{j,j+1}Q_{j,j+2}Q_{j+1,j+2},
\]
and
\[
U_{j+1}U_jU_{j+1}
=
P_{j+1}P_jP_{j+1}\,
Q_{j,j+1}Q_{j,j+2}Q_{j+1,j+2}.
\]
The conclusion follows from the usual braid relation
$
P_jP_{j+1}P_j
=
P_{j+1}P_jP_{j+1}$.
\end{proof}

We write
\[
J_{\beta,j}^{\perp}
=
\prod_{\substack{1\leq k\leq r\\ k\neq j,j+1}}
J_\beta^{(k)},
\]
with the convention that an empty product is the identity, and define
\begin{equation}
\label{eq:defAj}
\mathcal A_j= J_{\beta,j}^{\perp}P_jQ_j.
\end{equation}

\begin{proposition} 
With the notation as above, for $1\leq j<r$ and $u\in
(N_\beta^+)^{\widehat{\times} r}$, one has
\[
\mathcal A_j E(u)\mathcal A_j^{-1}
=
E(t_jut_j^{-1}).
\]
Moreover, the operators $\mathcal A_j$ satisfy the ordinary Coxeter
relations. Hence, they determine operators $\mathcal A_\sigma$ for
$\sigma\in\sym_r$, such that
\begin{equation}
\label{eq:relcoxAsigma}
\mathcal A_\sigma\mathcal A_\tau
=
\mathcal A_{\sigma\tau}.
\end{equation}
\end{proposition}

\begin{proof}
Set $U_j=P_jQ_j$.
By Lemma~\ref{lem:TD}, we have
\[
U_jR_\beta(x)^{(j)}U_j^{-1}
=
(J_\beta^{(j)})^{\partial(x)}R_\beta(x)^{(j+1)}
\]
and
\[
U_jR_\beta(x)^{(j+1)}U_j^{-1}
=
R_\beta(x)^{(j)}(J_\beta^{{j+1}})^{\partial(x)}.
\]
Moreover, we have
\[
U_jJ_\beta^{(j)}U_j^{-1}=J_\beta^{(j+1)}
\quad\text{and}\quad
U_jJ_\beta^{(j+1)}U_j^{-1}=J_\beta^{(j)},
\]
and $U_j$ commutes with $J_\beta^{(k)}$ for $k\neq j,j+1$. It follows that
\[
\mathcal A_jE\bigr(x^{(j)}\bigr)\mathcal A_j^{-1}
=
E\bigr(x^{(j+1)}\bigr)
\quad
\text{and}
\quad
\mathcal A_jE\bigr(x^{(j+1)}\bigr)\mathcal A_j^{-1}
=
E\bigr(x^{(j)}\bigr).
\]
If $k\neq j,j+1$, then $U_j$ commutes with $E(x^{(k)})$, 
and
\[
\begin{aligned}
\mathcal A_jE(x^{(k)})\mathcal A_j^{-1}
&=
J_{\beta,j}^{\perp}E(x^{(k)}) (J_{\beta,j}^{\perp})^{-1}\\
&=
(-1)^{\partial(x)}E(x^{(k)})\\
&=
E\bigl(z^{\partial(x)}x^{(k)}\bigr).
\end{aligned}
\]
Since $U_j$ commutes with $J_{\beta,j}^{\perp}$, we have $ \mathcal A_j^2
=I$. Using Lemma~\ref{lem:coxUi} and $
U_jJ_\beta^{(j)}U_j^{-1}=J_\beta^{(j+1)}$ and
$U_jJ_\beta^{(j+1)}U_j^{-1}=J_\beta^{(j)}$, we obtain
\[
\mathcal A_j\mathcal A_{j+1}\mathcal A_j
=
\mathcal A_{j+1}\mathcal A_j\mathcal A_{j+1}
\quad\text{and}\quad
\mathcal A_j\mathcal A_k
=
\mathcal A_k\mathcal A_j
\quad
\text{for }|j-k|>1.
\]
Hence, the $\mathcal A_j$ satisfy the same Coxeter relations as the transpositions
$s_j$. Therefore, by the Coxeter presentation of $\sym_r$, there is a
unique representation
\[
\mathcal A:\sym_r\longrightarrow\GL(V^{\otimes r})
\]
such that
\[
\mathcal A(s_j)=\mathcal A_j
\qquad
(1\leq j<r).
\]
For $\sigma\in\sym_r$, we write
\[
\mathcal A_\sigma=\mathcal A(\sigma).
\]
In particular, $\mathcal A_\sigma\mathcal A_\tau = \mathcal A_{\sigma\tau}
$ for all $\sigma,\tau\in\sym_r$,  and the result follows.
\end{proof}

For each $\sigma\in\sym_r$, choose once and for all a reduced expression
$\sigma=s_{i_1}\cdots s_{i_\ell}$, and define $ t_\sigma=t_{i_1}\cdots
t_{i_\ell}$.
The element $t_\sigma$ depends in general on the chosen reduced
expression, up to multiplication by $z$, and the set
\begin{equation}
\label{eq:transversal}
\{t_\sigma\mid \sigma\in\sym_r\}
\end{equation}
is a transversal for $\langle z\rangle$ in the corresponding copy of
$S_{\beta,r}^+$. 
For $\sigma,\,\tau\in\sym_r$, there is $c(\sigma,\tau)\in \langle
z\rangle$ such that
\[
t_\sigma t_\tau
=
c(\sigma,\tau)t_{\sigma\tau}.
\]

\begin{corollary}
The Clifford factor set associated with $E$, the $\mathcal A_\sigma$ and
the transversal \eqref{eq:transversal} is
$$\alpha(\sigma,\tau)=E(c(\sigma,\tau))^{-1}.$$
\end{corollary}

\begin{proof}

The Clifford factor set $\alpha$ associated with the operators $\mathcal
A_\sigma$ is characterized by
\[
\mathcal A_\sigma\mathcal A_\tau
=
\alpha(\sigma,\tau)
E\bigl(c(\sigma,\tau)\bigr)
\mathcal A_{\sigma\tau}.
\]
By \eqref{eq:relcoxAsigma}, we obtain
$
\alpha(\sigma,\tau)
E\bigl(c(\sigma,\tau)\bigr)=1.
$
The result follows.
\end{proof}

\begin{corollary}
\label{cor:paramSA}
The irreducible representations lying over $E$ are parametrized by the
strict partitions of $r$. More precisely, if $\lambda\in\mathcal D_r^+$,
there is a unique such representation, denoted by $\rho_{R_\beta,\lambda}$,
whereas if $\lambda\in\mathcal D_r^-$, there are two associate
representations, denoted by
$\rho_{R_\beta,\lambda}^+$ and $\rho_{R_\beta,\lambda}^-$.
\end{corollary}

\begin{proof}
Let $\rho_\lambda$, or $\rho_\lambda^\pm$, denote the usual spin
representations of $\tSym_r$ parametrized by $\lambda$. Let
$\widetilde t_1,\ldots,\widetilde t_{r-1}$ denote the standard generators
of $\tSym_r$. The rule
\[
\rho_\lambda^{(\beta)}(t_j)
=
\gamma_\beta\,\rho_\lambda(\widetilde t_j),
\qquad
\rho_\lambda^{(\beta)}(z)=-I,
\]
and similarly in the non-self-associate case, defines the corresponding
spin representations of $S_{\beta,r}^+$. Indeed, for
$p^\beta\equiv-1\pmod4$, multiplication of the simple generators by
$\mathrm i$ changes
\[
\widetilde t_j^2=z,\qquad
(\widetilde t_j\widetilde t_{j+1})^3=z
\]
into the defining relations
\[
t_j^2=1,\qquad
(t_jt_{j+1})^3=1,
\]
and leaves the last relations unchanged. Thus, the same bar
partitions parametrize the spin representations of $S_{\beta,r}^+$.
In either case, write $\rho$ for one of these representations of
$S_{\beta,r}^+$ and set
\begin{equation}
\label{eq:defR}
R(\sigma)=\rho(t_\sigma)
\quad
\text{for }\sigma\in\sym_r.
\end{equation}
Since $ t_\sigma t_{\sigma'} = c(\sigma,\sigma')t_{\sigma\sigma'} $, we
have
\[
\begin{aligned}
R(\sigma)R(\tau)
&=
\rho\bigl(c(\sigma,\tau)\bigr)R(\sigma\tau)\\
&=
E\bigl(c(\sigma,\tau)\bigr)R(\sigma\tau)\\
&=
\alpha^{-1}(\sigma,\tau)R(\sigma\tau),
\end{aligned}
\]
because $c(\sigma,\tau)\in\{1,z\}$ and $ \rho(z)=E(z)=-I$, hence
$\rho(c(\sigma,\sigma'))=E(c(\sigma,\sigma'))$. The result follows from
the projective form of Clifford theory (see for example
\cite[Chapter~11]{isaacs}).
\end{proof}

\begin{remark}
\label{rk:notR}
With the above notation, if $\lambda\in\mathcal D_r^+$, then for $m$ in
the base group and $\sigma\in\sym_r$,
\begin{equation}
\label{eq:valmatriceproj}
\rho_{R_\beta,\lambda}(m t_\sigma)
=
E(m)\mathcal A_\sigma\otimes R_\lambda(\sigma).
\end{equation}
If $\lambda\in\mathcal D_r^-$, the same formula holds with
$\rho_{R_\beta,\lambda}$ and $R_\lambda$ replaced simultaneously by
$\rho_{R_\beta,\lambda}^\pm$ and $R_\lambda^\pm$, respectively.
\end{remark}

\begin{corollary}
\label{cor:assogen}
Assume that $\lambda\in\mathcal D_r^+$, and let $J_\lambda$ be an
associator for the spin representation $\rho_\lambda$ of $S_{\beta,r}^+$.
Then $\rho_{R_\beta,\lambda}$ is self-associate and admits the associator
\[
J_{\beta,\lambda}
=
J_E\otimes J_\lambda,
\]
where $J_E=J_\beta^{(1)}\cdots J_\beta^{(r)}$. 
\end{corollary}

\begin{proof}
For $x\in N_\beta^+$, the definition of $E$ and
Equation~\eqref{eq:assoD} give
$
J_EE(x^{(j)})J_E^{-1}
=
\varepsilon(x)E(x^{(j)})$.
The operators $J_{\beta,j}^\perp$ commute with $J_E$ and $U_j$
interchanges $J_\beta^{(j)}$ and $J_\beta^{(j+1)}$, and fixes all the
other factors. It follows that $J_E$ commutes with $\mathcal A_j$, and
thus, with $\mathcal A_\sigma$. 
Using Equation~\eqref{eq:valmatriceproj} and $
J_\lambda\rho_\lambda(t_\sigma)J_\lambda^{-1}
=
\sgn(\sigma)\rho_\lambda(t_\sigma)$, we obtain
\[
\begin{aligned}
J_{\beta,\lambda}
\rho_{R_\beta,\lambda}(m t_\sigma)
J_{\beta,\lambda}^{-1}
&=
\varepsilon(m)E(m)\mathcal A_\sigma
 \otimes
\sgn(\sigma)R(\sigma)\\
&=
\varepsilon(m t_\sigma)
\rho_{R_\beta,\lambda}(m t_\sigma).
\end{aligned}
\]
Thus $J_{\beta,\lambda}$ is an associator for $\rho_{R_\beta,\lambda}$.
\end{proof}

\subsection{The intermediate Young subgroup and Clifford theory}

Let $\beta\geq1$.
Recall that $\rho_\beta^{(0)}=R_\beta$ denotes the unique self-associate
irreducible spin representation of $N_\beta^+$, and choose representatives
\[
\rho_\beta^{(1)},\ldots,\rho_\beta^{(e_\beta)}
\]
of the non-self-associate association classes of irreducible spin
representations of $N_\beta^+$.
Let
\begin{equation}
\label{eq:profile}
\mathbf n_\beta
=
\bigl(
n_\beta^{(0)},n_\beta^{(1)},\ldots,n_\beta^{(e_\beta)}
\bigr)
\quad\text{such that}\quad
\sum_{j=0}^{e_\beta}n_\beta^{(j)}=n_\beta,
\end{equation}
where $n_\beta$ is defined in \eqref{eq:expansionadicn}.
Write $M_\beta=N_\beta^{n_\beta}$. We choose an irreducible spin representation
$
\theta_{\mathbf n_\beta}\in\Irr(M_\beta^+)
$
whose homogeneous components have multiplicities
$n_\beta^{(0)},\ldots,n_\beta^{(e_\beta)}$, respectively. We denote its
inertia subgroup in $H_\beta^+$ by
\[
T_\beta^+
=
I_{H_\beta^+}(\theta_{\mathbf n_\beta}).
\]
Let
\[
Y_{\mathbf n_\beta}
=
\sym_{n_\beta^{(0)}}\times
\sym_{n_\beta^{(1)}}\times\cdots\times
\sym_{n_\beta^{(e_\beta)}}
\leq
\sym_{n_\beta},
\]
and let
$
q_\beta:H_\beta^+\longrightarrow\sym_{n_\beta}
$
be the natural permutation map. We set
\[
K_\beta=q_\beta^{-1}(Y_{\mathbf n_\beta}).
\]
Note that \cite[Propositions~3.12 and~3.13]{MichlerOlsson} yields
\[
T_\beta^+\leq K_\beta\leq H_\beta^+.
\]
Moreover, we have
\[
K_\beta
\simeq
\widehat{\prod}_{j=0}^{e_\beta}
\left(
N_\beta\wr\sym_{n_\beta^{(j)}}
\right)^+.
\]
Now, by Corollary~\ref{cor:paramSA}, the association class of the
$0$-th component, corresponding to the unique self-associate elementary
factor $\rho_\beta^{(0)}$, is parametrized by a bar partition
\[
\lambda_\beta^{(0)}\in\mathcal D_{n_\beta^{(0)}}.
\]
For $1\leq j\leq e_\beta$, the association class of the $j$-th homogeneous
component, corresponding to the non-self-associate elementary
association class represented by $\rho_\beta^{(j)}$, is parametrized by an
ordinary partition
\[
\lambda_\beta^{(j)}\vdash n_\beta^{(j)}.
\]
We set
\[
\mathcal C_\beta
=
\bigl(
\lambda_\beta^{(0)},\lambda_\beta^{(1)},\ldots,
\lambda_\beta^{(e_\beta)}
\bigr).
\]

\begin{lemma}
\label{lem:typeQi}
For every choice of the partitions
$\lambda_\beta^{(0)},\ldots,\lambda_\beta^{(e_\beta)}$
described above, the tuple $\mathcal C_\beta$ determines a unique
association class of irreducible spin representations of $K_\beta$.
This class is self-associate if and only if
\[
n_\beta-\ell(\lambda_\beta^{(0)})\equiv0\mod2.
\]
\end{lemma}

\begin{proof}
The $0$-th homogeneous factor is self-associate if and only if
$n_\beta^{(0)}-\ell(\lambda_\beta^{(0)})\equiv0\mod2$, and for $j\geq1$, the
$j$-th homogeneous factor is self-associate if and only if $n_\beta^{(j)}$ is
even.
Using Humphreys' rule for the twisted tensor product of spin
representations~\cite[Proposition 1.2]{MichlerOlsson}, the resulting
association class is self-associate if and only if
\[
n_\beta^{(0)}-\ell(\lambda_\beta^{(0)})
+
\sum_{j=1}^{e_\beta}n_\beta^{(j)}
\equiv0\mod2.
\]
Since $ \sum_{j=0}^{e_\beta}n_\beta^{(j)}=n_\beta $, this is equivalent to $
n_\beta-\ell(\lambda_\beta^{(0)})\equiv0\mod2$.
\end{proof}

\begin{notation}
In the self-associate case, we denote the representation by
$\Psi_{\beta,\mathcal C_\beta}$. In the non-self-associate case, we denote
the two associate representations by $\Psi_{\beta,\mathcal C_\beta}^+$ and
$\Psi_{\beta,\mathcal C_\beta}^-$.
\end{notation}

We shall use the following elementary consequence of Clifford theory.

\begin{lemma}
\label{lem:intermediate-clifford}
Let $M\triangleleft H$, let $\theta\in\Irr(M)$, and set
$
T=I_H(\theta)$.
Suppose that
\[
T\leq K\leq H.
\]
Then induction gives a bijection $ \Irr(T\mid\theta) \longrightarrow
\Irr(K\mid\theta), \ \varphi\longmapsto\Ind_T^K\varphi$. Moreover, if $
\Psi=\Ind_T^K\varphi $, then $ \Ind_K^H\Psi = \Ind_T^H\varphi $ is
irreducible.
\end{lemma}

\begin{proof}
Since $ I_K(\theta) = I_H(\theta)\cap K = T$, the first assertion is the
Clifford correspondence \cite[Theorem~(6.11)]{isaacs}. Thus, for every
$\Psi\in\Irr(K\mid\theta)$, there is a unique
$\varphi\in\Irr(T\mid\theta)$ such that
$
\Psi=\Ind_T^K\varphi$. 
Again by the Clifford correspondence,
$
\Ind_T^H\varphi\in\Irr(H)
$.
Finally, by the transitivity of induction, we obtain
\[
\Ind_K^H\Psi
=
\Ind_K^H\Ind_T^K\varphi
=
\Ind_T^H\varphi,
\]
as required.
\end{proof}

\begin{lemma}
\label{lem:type-induction}
Induction from $K_\beta$ to $H_\beta^+$ preserves the association type.
More precisely, if $\Psi$ is one of the irreducible spin representations
of $K_\beta$ constructed above, then
\[
\Psi\text{ is self-associate}
\quad\Longleftrightarrow\quad
\Ind_{K_\beta}^{H_\beta^+}\Psi\text{ is self-associate}.
\]
In particular, if $\Psi^+$ and $\Psi^-$ are associate, then their
inductions to $H_\beta^+$ are distinct and associate.
\end{lemma}

\begin{proof}
Write $\varepsilon_M$, $\varepsilon_K$ and $\varepsilon_H$ for the
restrictions of the sign character to $M_\beta^+$, $K_\beta$ and
$H_\beta^+$, respectively. Since twisting by a linear character commutes
with induction,
\[
\varepsilon_H\Ind_{K_\beta}^{H_\beta^+}\Psi
=
\Ind_{K_\beta}^{H_\beta^+}(\varepsilon_K\Psi).
\]
Hence, if $\Psi$ is self-associate, so is its induction. For the converse,
set $ r=n_\beta^{(0)}$. Suppose first that $r\geq2$. Choose a simple
transposition $s_j$ whose two entries lie in the $0$-th block. Then
$t_j\in K_\beta$, and by \cite[Lemma~3.11]{MichlerOlsson}, conjugation by
$t_j$ exchanges the two corresponding copies of $R_\beta$ and replaces
each remaining elementary factor by its associate. Since $R_\beta$ is
self-associate, Humphreys' association rule
\cite[Proposition~1.2]{MichlerOlsson} therefore gives
\begin{equation}
\label{eq:equivloc}
\theta_{\mathbf n_\beta}^{\,t_j}
=
\varepsilon_M\theta_{\mathbf n_\beta}.
\end{equation}
Now, assume that $ \chi=\Ind_{K_\beta}^{H_\beta^+}\Psi $ is
self-associate. Then $ \Ind_{K_\beta}^{H_\beta^+}\Psi \simeq
\Ind_{K_\beta}^{H_\beta^+}(\varepsilon_K\Psi) $.
Since $t_j\in K_\beta$, we have $\Psi^{t_j}=\Psi$. Hence
\eqref{eq:equivloc} shows that $\varepsilon_M\theta_{\mathbf n_\beta}$ is a
constituent of $\Res_{M_\beta^+}^{K_\beta}\Psi$. Thus both $\Psi$ and
$\varepsilon_K\Psi$ lie over $\varepsilon_M\theta_{\mathbf n_\beta}$.
The injectivity of the Clifford correspondence then gives
\[
\Psi\simeq\varepsilon_K\Psi,
\]
so that $\Psi$ is self-associate. Suppose now that $r\leq1$. Set
\[
u=n_\beta-r=\sum_{j=1}^{e_\beta}n_\beta^{(j)}.
\]
If $u$ is even, \cite[Proposition 1.2]{MichlerOlsson} shows that $\Psi$ is
self-associate. Since twisting by the sign character commutes with
induction, $\Ind_{K_\beta}^{H_\beta^+}\Psi$ is self-associate as well.
Assume now that $u$ is odd. Then $\Psi$ is non-self-associate, and
\cite[Proposition~3.12]{MichlerOlsson} shows that $\theta_{\mathbf
n_\beta}$ and $\varepsilon_M\theta_{\mathbf n_\beta}$ belong to distinct
$H_\beta^+$-orbits. By Clifford theory, the irreducible constituents of
$\Res_{M_\beta^+}^{H_\beta^+}\chi$ form the orbit of $\theta_{\mathbf
n_\beta}$, whereas those of
$\Res_{M_\beta^+}^{H_\beta^+}(\varepsilon_H\chi)$ form the orbit of
$\varepsilon_M\theta_{\mathbf n_\beta}$. Hence $\chi$ cannot be
self-associate.
This proves the lemma.
\end{proof}

We apply Lemma~\ref{lem:intermediate-clifford} with $ M=M_\beta^+$, $
H=H_\beta^+$, $\theta=\theta_{\mathbf n_\beta}$, $T=T_\beta^+$ and $
K=K_\beta$.
Hence, induction from $K_\beta$ to $H_\beta^+$ sends each of the representations
constructed above to an irreducible spin representation of $H_\beta^+$.
\begin{notation}
\label{not:RiQi}
In the self-associate case, we set
\[
\rho_{\beta,\mathcal C_\beta}
=
\Ind_{K_\beta}^{H_\beta^+}\Psi_{\beta,\mathcal C_\beta}.
\]
In the non-self-associate case we set
\[
\rho_{\beta,\mathcal C_\beta}^\pm
=
\Ind_{K_\beta}^{H_\beta^+}\Psi_{\beta,\mathcal C_\beta}^\pm.
\]
\end{notation}

\begin{notation}
\label{not:admissible-beta-tuples}
A tuple
\[
\mathcal C_\beta
=
\bigl(
\lambda_\beta^{(0)},\lambda_\beta^{(1)},\ldots,
\lambda_\beta^{(e_\beta)}
\bigr)
\]
is called an \emph{admissible $\beta$-tuple} if $\lambda_\beta^{(0)}$ is a
bar partition, $\lambda_\beta^{(j)}$ is an ordinary partition for $1\leq
j\leq e_\beta$, and
\[
\sum_{j=0}^{e_\beta}|\lambda_\beta^{(j)}|=n_\beta.
\]
We denote by $\mathfrak C_\beta$ the set of all admissible
$\beta$-tuples.
\end{notation}

\begin{theorem}
\label{thm:param-level}
The map which associates with $ \mathcal C_\beta\in\mathfrak C_\beta $ the
association class represented by $\rho_{\beta,\mathcal C_\beta}$, or by
$\rho_{\beta,\mathcal C_\beta}^\pm$ in the non-self-associate case, is a
bijection between $\mathfrak C_\beta$ and the association classes of
irreducible spin representations of $H_\beta^+$. Moreover, the class
corresponding to $\mathcal C_\beta$ is self-associate if and only if
\[
n_\beta-\ell(\lambda_\beta^{(0)})\equiv0\mod2.
\]
\end{theorem}
\begin{proof}
By the preceding construction and the Clifford correspondence of
Lemma~\ref{lem:intermediate-clifford}, the irreducible spin representations
of $H_\beta^+$ are parametrized, up to association, by the admissible tuples
$\mathcal C_\beta\in\mathfrak C_\beta$. The statement on association type
follows from Lemmas~\ref{lem:typeQi} and~\ref{lem:type-induction}.
\end{proof}

\section{Galois action on the local parametrization}
\label{sec:galois-local}

\subsection{Strategy}

In this section, we keep the notation of \S\ref{subsec:indice2}. Thus,
$H\triangleleft G$ is a normal subgroup of index $2$, and $
\varepsilon:G\longrightarrow\{\pm1\} $ denotes the non-trivial linear
character of $G/H$ with kernel $H$. The association classes in $\Irr(G)$
are labeled by a set $\Lambda$. Let $\lambda\in\Lambda$, and let $\tau$ be
a Galois automorphism stabilizing the association class labelled by
$\lambda$. We describe a general strategy for reading the action of $\tau$
from a suitable intertwining operator. 
\smallskip

We denote by $\rho_\lambda$, or by $\rho_\lambda^+$ and
$\rho_\lambda^-$, representations affording the corresponding characters
$\chi_\lambda$, or $\chi_\lambda^+$ and $\chi_\lambda^-$, respectively.
Thus $\rho_\lambda$ is a representation of $G$ in the self-associate
case and of $H$ in the non-self-associate case.

Since $\tau$ stabilizes the association class labelled by $\lambda$, we
have $ \rho_\lambda^\tau\simeq\rho_\lambda $. We choose an intertwiner
$
B:\rho_\lambda^\tau\longrightarrow\rho_\lambda,
$
that is 
$$
B\rho_\lambda^\tau B^{-1}=\rho_\lambda.
$$

\begin{enumerate}[(i)]
\item Suppose that $\lambda$ labels an orbit of size $1$. We choose an
associator $J$ for $\rho_\lambda$ such that, for all $g\in G$,
\[
J\rho_\lambda(g)J^{-1}
=
\varepsilon(g)\rho_\lambda(g)
\quad\text{and}\quad
J^2=I.
\]

\item Suppose that $\lambda$ labels an orbit of size $2$. We realize the
two extensions $\rho_\lambda^\pm$ on the space of $\rho_\lambda$, with $
\rho_\lambda^-=\varepsilon\rho_\lambda^+ $. Choose $x\in G\setminus H$ and
set
\[
A=\rho_\lambda^+(x).
\]
\end{enumerate}

\begin{lemma}
\label{lem:galois-index-two}
With the notation above, there exists
\[
c_\tau(\lambda)\in\{\pm1\}
\]
such that  
$
BJ^\tau B^{-1}
=
c_\tau(\lambda)J
$
when the orbit has size $1$, and 
$
BA^\tau B^{-1}
=
c_\tau(\lambda)A
$
otherwise.
In either case, for $\eta\in\{\pm1\}$,
\[
(\rho_\lambda^\eta)^\tau
\simeq
\rho_\lambda^{\,c_\tau(\lambda)\eta}.
\]
\end{lemma}

\begin{proof}
Suppose first that $\lambda$ labels an orbit of size $1$. Then
$BJ^\tau B^{-1}$ is again an associator for $\rho_\lambda$. Hence, by
Schur's lemma,
$
BJ^\tau B^{-1}=cJ
$
for some $c\in\mathbb C^\times$. Since $J^2=I$ and $(J^\tau)^2=I$, we
obtain $c^2=1$.
Let $\eta\in\{\pm1\}$, and let $V_\tau^\eta$ denote the $\eta$-eigenspace
of $J^\tau$. Then for $v\in V_\tau^\eta$, we have
\[
J(Bv)
=
cBJ^\tau v
=
c\eta\,Bv.
\]
It follows that
\[
B(V_\tau^\eta)=V^{c\eta},
\]
where $V^{c\eta}$ is the $(c\eta)$-eigenspace of $J$. Since these
eigenspaces afford the two irreducible constituents of the restriction
to $H$, we deduce that
\[
(\rho_\lambda^\eta)^\tau
\simeq
\rho_\lambda^{c\eta}.
\]

Suppose now that $\lambda$ labels an orbit of size $2$. 
For $h\in H$, we have
\[
A\rho_\lambda(h)A^{-1}
=
\rho_\lambda(xhx^{-1}).
\]
Hence, $BA^\tau B^{-1}$ and $A$ satisfy the same intertwining relation,
and Schur's lemma gives
$
BA^\tau B^{-1}=cA
$
for some $c\in\mathbb C^\times$. Now,
$
A^2=\rho_\lambda(x^2)$, 
and 
\[
(BA^\tau B^{-1})^2
=
B\rho_\lambda^\tau(x^2)B^{-1}
=
\rho_\lambda(x^2)
=
A^2.
\]
We again obtain $c^2=1$ since $A$ is invertible.
Finally, for $\eta\in\{\pm1\}$ we have
\[
\rho_\lambda^\eta(h)=\rho_\lambda(h)
\quad\text{for }h\in H\quad\text{and}\quad
\rho_\lambda^\eta(x)=\eta A.
\]
Hence $B$ intertwines
$(\rho_\lambda^\eta)^\tau$ with $\rho_\lambda^{c\eta}$, and the result
follows.
\end{proof}

\begin{notation}
\label{def:galois-bit}
With the notation of Lemma~\ref{lem:galois-index-two}, for
$\lambda\in\Lambda$, define $ e_\tau(\lambda)\in\mathbb Z/2\mathbb Z $ by
\[
c_\tau(\lambda)=(-1)^{e_\tau(\lambda)}.
\]
In particular, $e_\tau(\lambda)=0$ if $\tau$ fixes the two representations
individually, and $e_\tau(\lambda)=1$ if it exchanges them.

If $V\in\Irr(G)$ belongs to the association class labelled by
$\lambda$, we also write
\[
e_\tau(V)=e_\tau(\lambda).
\]
\end{notation}

\subsection{Propagation through Humphreys products}

\begin{lemma}
\label{lem:BN-propagation}
Let $V_i$ be irreducible spin representations entering a finite Humphreys
product, such that their association classes are stable under $\tau$, and
assume that $\tau(i)=i$. Then, with the notation~\ref{def:galois-bit}, we
have
\[
e_\tau(V_1\widehat\otimes\cdots\widehat\otimes V_r)
\equiv
\sum_{k=1}^r e_\tau(V_k)
\mod2.
\]
\end{lemma}

\begin{proof}
This follows directly from the character formulas for Humphreys products
in \cite[Section~5, formulas~(27)--(35)]{BrNaMZ}. In the SA/SA and mixed
cases, these formulas show that the relevant character or character
difference factors as the product of the corresponding data of the two
factors, so that the Galois signs multiply. In the NSA/NSA case,
formula~(35) contains the additional factor $2i$. Since $\tau(i)=i$,
this factor contributes no extra sign. The binary formula follows, and
the general case follows by iteration.
\end{proof}

\subsection{Homogeneous factors}

Let $\beta\geq 1$ and $r\geq 1$. Let $\tau=\tau_p$ be a generator of
$G(p)$. 

\subsubsection{The self-associate elementary factor}

\begin{lemma}
\label{lem:galois-E}
Let  $E$ be the representation of $(N_\beta^+)^{\widehat{\times} r}$
defined in~\eqref{eq:defESA}. There exists an intertwiner $
B_\beta:R_\beta^\tau\longrightarrow R_\beta $
such that
\begin{equation}
\label{eq:conjassD}
B_\beta J_\beta^\tau B_\beta^{-1}
=
(-1)^\beta J_\beta.
\end{equation}
Moreover, setting
\[
C_{\beta,r}
=
\prod_{k=1}^r
\bigl(J_\beta^{(k)}\bigr)^{\beta(k-1)}
\quad\text{and}\quad
\mathcal B_{\beta,r}
=
C_{\beta,r}B_\beta^{\otimes r},
\]
one has, for every $u\in(N_\beta^+)^{\widehat{\times} r}$
\begin{equation}
\label{eq:interbase}
\mathcal B_{\beta,r}E(u)^\tau
\mathcal B_{\beta,r}^{-1}
=
E(u).
\end{equation}
\end{lemma}

\begin{proof}
Recall the explicit realization of the unique self-associate spin
representation $\varphi$ of $G_p$, with basis
$(e_u)_{u\in\mathbb F_p^\times}$ and associator $J$ defined
in~\eqref{eq:defassbase}. Let $\kappa\in\mathbb F_p^\times$ be defined by
\[
\tau(\zeta_p)=\zeta_p^\kappa,
\]
and define
\[
Be_u=e_{\kappa u}.
\]
By the explicit realization of $\varphi$, the operator $B$ intertwines
$\varphi^\tau$ and $\varphi$, so that
\[
B\varphi^\tau B^{-1}=\varphi.
\]
Since the restriction of $\tau$ to $\Q_{p(p-1)}$ generates
\[
\Gal(\Q_{p(p-1)}/\Q_{p-1})
\simeq\mathbb F_p^\times,
\]
the element $\kappa$ generates $\mathbb F_p^\times$. Hence
$\eta(\kappa)=-1$.
Since the values of $\eta$ belong to $\{\pm1\}$, we have $J^\tau=J$.
Moreover, for $u\in\mathbb F_p^\times$, we have
\[
BJ^\tau B^{-1}e_u
=
BJ e_{\kappa^{-1}u}
=
\eta(\kappa^{-1}u)e_u
=
\eta(\kappa)\eta(u)e_u
=
-Je_u.
\]
Hence,
$
BJ^\tau B^{-1}=-J$.
Now, recall that
$
R_\beta=\Lambda_\beta\otimes\varphi^{\boxtimes\beta}$
and 
$J_\beta=J^{\otimes\beta}$. Furthermore, 
$\Lambda_\beta$ is fixed by $\tau$. Thus, with
$
B_\beta=B^{\otimes\beta}$,
we obtain
\[
B_\beta R_\beta^\tau(x)B_\beta^{-1}
=
R_\beta(x)
\quad
\text{for }x\in N_\beta^+,
\]
and
\[
B_\beta J_\beta^\tau B_\beta^{-1}
=
\bigotimes_{k=1}^\beta
\bigl(BJ^\tau B^{-1}\bigr)
=
(-1)^\beta J_\beta.
\]

It remains to construct an intertwiner for $E$. It is enough to check
the required identity on the generators $x^{(j)}$, with
$x\in N_\beta^+$ and $1\leq j\leq r$. Put $d=\partial(x)$. By
\eqref{eq:defESA} and \eqref{eq:conjassD}, we obtain
\[
B_\beta^{\otimes r}
E(x^{(j)})^\tau
(B_\beta^{-1})^{\otimes r}
=
(-1)^{\beta(j-1)d}E(x^{(j)}).
\]
On the other hand, all the factors of $C_{\beta,r}$ commute with
$E(x^{(j)})$, except possibly the factor in position $j$. Since
\[
J_\beta^{(j)}
R_\beta(x)^{(j)}
\bigl(J_\beta^{(j)}\bigr)^{-1}
=
(-1)^dR_\beta(x)^{(j)},
\]
we get
\[
C_{\beta,r}E(x^{(j)})C_{\beta,r}^{-1}
=
(-1)^{\beta(j-1)d}E(x^{(j)}).
\]
The two signs therefore cancel, and
\[
\mathcal B_{\beta,r}
E(x^{(j)})^\tau
\mathcal B_{\beta,r}^{-1}
=
E(x^{(j)}).
\]
This proves the result.
\end{proof}

\begin{lemma}
\label{lem:galois-A}
Assume that $r\geq2$. For every $1\leq j<r$, one has
\[
\mathcal B_{\beta,r}\mathcal A_j^\tau
\mathcal B_{\beta,r}^{-1}
=
(-1)^{\beta(r-1)}\mathcal A_j.
\]
Moreover, for every $\sigma\in\sym_r$,
\begin{equation}
\label{eq:intergal}
\mathcal B_{\beta,r}\mathcal A_\sigma^\tau
\mathcal B_{\beta,r}^{-1}
=
\sgn(\sigma)^{\beta(r-1)}\mathcal A_\sigma.
\end{equation}
\end{lemma}

\begin{proof}
Set
$ B_0=B_\beta^{\otimes r}$ and recall that
\[
Q_j
=
\frac12\Bigl(
I+J_\beta^{(j)}+J_\beta^{(j+1)}
-J_\beta^{(j)}J_\beta^{(j+1)}
\Bigr).
\]
By~\eqref{eq:conjassD},
\[
\begin{aligned}
B_0Q_j^\tau B_0^{-1}
&=
\frac12\Bigl(
I+(-1)^\beta J_\beta^{(j)}
+(-1)^\beta J_\beta^{(j+1)}
-J_\beta^{(j)}J_\beta^{(j+1)}
\Bigr)\\
&=
(-1)^\beta
\bigl(J_\beta^{(j)}J_\beta^{(j+1)}\bigr)^\beta
Q_j.
\end{aligned}
\]
On the other hand, $P_j$ is fixed by $\tau$, and commutes with $B_0$ and
with $Q_j$.
Since $U_j=P_jQ_j$, it follows that
\begin{equation}
\label{eq:galois-Tj}
B_0U_j^\tau B_0^{-1}
=
(-1)^\beta
\bigl(J_\beta^{(j)}J_\beta^{(j+1)}\bigr)^\beta
U_j.
\end{equation}
Furthermore, Equation~\eqref{eq:conjassD} also yields
\[
B_0\bigl(J_{\beta,j}^{\perp}\bigr)^\tau B_0^{-1}
=
(-1)^{\beta(r-2)}J_{\beta,j}^{\perp}.
\]
Since $J_{\beta,j}^{\perp}$ commutes with $U_j$, we deduce from
\eqref{eq:galois-Tj} that
\begin{equation}
\label{eq:galois-A-before-C}
B_0\mathcal A_j^\tau B_0^{-1}
=
(-1)^{\beta(r-1)}
\bigl(J_\beta^{(j)}J_\beta^{(j+1)}\bigr)^\beta
\mathcal A_j.
\end{equation}

It remains to conjugate by $C_{\beta,r}$. Since $Q_j$ commutes with all
the $J_\beta^{(k)}$, it commutes with $C_{\beta,r}$. Moreover, $P_j$
exchanges the $j$th and $(j+1)$st tensor positions, then
\[
C_{\beta,r}P_jC_{\beta,r}^{-1}
=
\bigl(J_\beta^{(j)}J_\beta^{(j+1)}\bigr)^\beta P_j.
\]
Hence,
\[
C_{\beta,r}U_jC_{\beta,r}^{-1}
=
\bigl(J_\beta^{(j)}J_\beta^{(j+1)}\bigr)^\beta U_j,
\]
and, since $C_{\beta,r}$ commutes with $J_{\beta,j}^{\perp}$, we obtain
\begin{equation}
\label{eq:conj-C-Aj}
C_{\beta,r}\mathcal A_jC_{\beta,r}^{-1}
=
\bigl(J_\beta^{(j)}J_\beta^{(j+1)}\bigr)^\beta
\mathcal A_j.
\end{equation}
Combining~\eqref{eq:galois-A-before-C} and
\eqref{eq:conj-C-Aj}, and using
$\mathcal B_{\beta,r}=C_{\beta,r}B_0$, gives
\[
\mathcal B_{\beta,r}\mathcal A_j^\tau
\mathcal B_{\beta,r}^{-1}
=
(-1)^{\beta(r-1)}\mathcal A_j.
\]
Finally, Equation \eqref{eq:relcoxAsigma} gives
$
\mathcal B_{\beta,r}\mathcal A_\sigma^\tau
\mathcal B_{\beta,r}^{-1}
=
\sgn(\sigma)^{\beta(r-1)}\mathcal A_\sigma
$,
as claimed.
\end{proof}

For fixed $\beta$, we shall also regard $\lambda\in\mathcal D_r$ as the
label of the association class corresponding to $\lambda$ in
Corollary~\ref{cor:paramSA}.

\begin{lemma}
\label{lem:galois-D}
Assume that $r<p$. Let $\lambda\in\mathcal D_r$. The representations
corresponding to $\lambda$ in Corollary~\ref{cor:paramSA} are stable under
$\tau$, and
\[
e_\tau(\lambda)
\equiv
\beta\,\ell(\lambda)
\mod2.
\]
\end{lemma}

\begin{proof}
Set
\[
d\equiv \beta(r-1)\mod2.
\]
Assume first that $\lambda\in\mathcal D_r^+$. By
Corollary~\ref{cor:assogen}, the representation
$\rho_{R_\beta,\lambda}$ is self-associate. 
Let
\[
A_{\beta,r}^+=\ker\bigl(\varepsilon|_{S_{\beta,r}^+}\bigr).
\]
Set $K=\Q_{2r!}$. Since $r<p$, the field $K$ is fixed by $\tau$, and
by Brauer's theorem \cite[Theorem (10.3)]{isaacs}, it is a splitting
field for both $S_{\beta,r}^+$ and $A_{\beta,r}^+$. We may choose a
$K$-model $V_{\lambda,K}$ of $\rho_\lambda$. Since $\rho_\lambda$ is
self-associate, Clifford theory over $K$ gives
\[
V_{\lambda,K}
=
V_{\lambda,K}^+\oplus V_{\lambda,K}^-.
\]
where $V_{\lambda,K}^+$ and $V_{\lambda,K}^-$ are non-isomorphic
irreducible $\widetilde{\alt}_r$-modules. After extension of scalars to
$\C$, they afford the two constituents $\rho_\lambda^+$ and
$\rho_\lambda^-$ of the restriction of $\rho_\lambda$.
Now, we define $J_\lambda$ on $V_{\lambda,K}$
by
\[
J_\lambda|_{V_{\lambda,K}^+}=I,
\qquad
J_\lambda|_{V_{\lambda,K}^-}=-I.
\]
Then $J_\lambda^2=I$ and $ J_\lambda \rho_\lambda(g)J_\lambda^{-1} =
\varepsilon(g)\rho_\lambda(g) $. By \eqref{eq:defR}, we deduce that for
all $\sigma\in \sym_r$,
\begin{equation}
\label{eq:locrel}
J_\lambda R_\lambda(\sigma)
J_\lambda^{-1}=\sgn(\sigma)R_\lambda(\sigma).
\end{equation}
Moreover, since $J_\lambda\in \GL(V_{\lambda,K})$ and $\tau$ acts trivially
on $K$, we deduce that $ J_\lambda^\tau=J_\lambda $.

\noindent
Set \[
B_{\beta,\lambda}
=
\mathcal B_{\beta,r}\otimes J_\lambda^d.
\]
Let $m\in (N_\beta^+)^{\widehat{\times} r}$ and $\sigma\in\sym_r$. 
Now, Equation \eqref{eq:defR} gives $R_\lambda^\tau= R_\lambda$. By
\eqref{eq:valmatriceproj}, \eqref{eq:interbase}, \eqref{eq:intergal} and
\eqref{eq:locrel} we
obtain
\[
\begin{aligned}
B_{\beta,\lambda}
\rho_{R_\beta,\lambda}(mt_\sigma)^\tau
B_{\beta,\lambda}^{-1}
&=
E(u)\,
\sgn(\sigma)^d\mathcal A_\sigma
\otimes
J_\lambda^dR_\lambda(\sigma)J_\lambda^{-d}\\
&=
\sgn(\sigma)^{2d}E(u)\mathcal A_\sigma\otimes R_\lambda(\sigma)\\
&=\rho_{R_\beta,\lambda}(m t_\sigma).
\end{aligned}
\]
Hence, $B_{\beta,\lambda}$ intertwines $\rho_{R_\beta,\lambda}^\tau$ and
$\rho_{R_\beta,\lambda}$. Moreover, since $C_{\beta,r}$ commutes with
$J_E$, equation~\eqref{eq:conjassD} gives
\[
\mathcal B_{\beta,r}J_E^\tau\mathcal B_{\beta,r}^{-1}
=
(-1)^{\beta r}J_E.
\]
Since $J_\lambda^d$ commutes with $J_\lambda$ and
$J_\lambda^\tau=J_\lambda$, we then obtain
\[
B_{\beta,\lambda}
J_{\beta,\lambda}^\tau
B_{\beta,\lambda}^{-1}
=
(-1)^{\beta r}J_{\beta,\lambda}.
\]
Hence, by Lemma~\ref{lem:galois-index-two} we have
\[
e_\tau(\lambda)
\equiv\beta r\mod2.
\]
Moreover, by definition of $\mathcal D_r^+$, we have
$r-\ell(\lambda)\equiv0\mod2$. Finally,
\[
e_\tau(\lambda)
\equiv
\beta\,\ell(\lambda)
\mod2.
\]

Assume now that $\lambda\in\mathcal D_r^-$. By Remark~\ref{rk:notR}, the
corresponding association class consists of a non-self-associate pair $
\rho_{R_\beta,\lambda}^+$  and $\rho_{R_\beta,\lambda}^-$.
Now, using Equation~\eqref{eq:valmatriceproj} and
$(R_\lambda^+)^\tau=R_\lambda^+$,
we have
\[
\begin{aligned}
(\mathcal B_{\beta,r}\otimes I)
\rho_{R_\beta,\lambda}^+(ut_\sigma)^\tau
(\mathcal B_{\beta,r}^{-1}\otimes I)
&=
E(u)\,
\sgn(\sigma)^d\mathcal A_\sigma
\otimes R_\lambda^+(\sigma).
\end{aligned}
\]
Thus $(\rho_{R_\beta,\lambda}^+)^\tau$ is equivalent to
$\rho_{R_\beta,\lambda}^+$ if $d=0$, and to
$\rho_{R_\beta,\lambda}^-$ if $d=1$ (because
$R_\lambda^-(\sigma)=-R_\lambda^+(\sigma)$). This proves that
\[
e_\tau(\lambda)
\equiv d
\equiv\beta(r-1)
\mod2.
\]
Finally, since $\lambda\in\mathcal D_r^-$, $ r-\ell(\lambda)\equiv1\mod2
$, so that
\[
r-1\equiv\ell(\lambda)\mod2.
\]
Hence,
\[
e_\tau(\lambda)
\equiv
\beta\,\ell(\lambda)
\mod2.
\]
This proves the claim.
\end{proof}

\subsubsection{The non-self-associate elementary factors}

\begin{lemma}
\label{lem:galois-NSA-elementary}
Let $\rho$ be a non-self-associate elementary spin representation of
$N_\beta^+$. For every partition $\lambda\vdash r$, let $E_{\rho,\lambda}$
be the corresponding irreducible representation of $(N_\beta\wr \sym_r)^+$.
Then $E_{\rho,\lambda}$ is fixed by $\tau$ and
\[
e_\tau(E_{\rho,\lambda})=0.
\]
\end{lemma}

\begin{proof}
Recall that $\rho$ has the form
$
\rho
=
\Lambda_\beta\otimes
\Inf(\phi_1\boxtimes\cdots\boxtimes\phi_\beta)$,
where each $\phi_k$ is either linear or equal to $\varphi$. Let $B$ be the
intertwiner previously defined for $\varphi$, so that $ B\varphi^\tau
B^{-1}=\varphi$. For $1\leq k\leq\beta$, set
\[
B_k=
\begin{cases}
I,&\text{if $\phi_k$ is linear},\\
B,&\text{if $\phi_k=\varphi$},
\end{cases}
\quad\text{and}\quad
B_\rho=B_1\otimes\cdots\otimes B_\beta.
\]
The linear factors, as well as $\Lambda_\beta$, are fixed by $\tau$.
Hence, for $ x\in N_\beta^+$,
\[
B_\rho\rho^\tau(x)B_\rho^{-1}
=
\rho(x).
\]
Recall that the homogeneous representation $E_\rho$ is realized on
$V^{\otimes r}\otimes W_r$.
Set
\[
\mathcal B_\rho
=
B_\rho^{\otimes r}\otimes I_{W_r}.
\]
Schur's explicit construction \cite[\S22, p. 209]{schur} shows that the matrices
\(F_\nu\) have entries in \(\mathbb Q(i)\), and the normalization used by
Michler--Olsson differs only by fourth roots of unity; hence the matrices
\(F_\nu\) used here are defined over \(\mathbb Q(i)\). In particular, $
F_\nu^{\tau_p}=F_\nu$ for all $\nu$.
It follows that
\[
\mathcal B_\rho
E_\rho(x^{(j)})^\tau
\mathcal B_\rho^{-1}
=
E_\rho(x^{(j)}).
\]
Moreover, $B_\rho^{\otimes r}$ commutes with $P_j$, since the same
operator $B_\rho$ acts on every tensor factor. 
Michler--Olsson use
\[
D_j
=
(-1)^{r-j-1}\sqrt{-\frac12}\,
\bigl(F_{r-j}+F_{r-j+1}\bigr),
\]
which are fixed by $\tau$. Since
$\gamma_\beta\in\{1,\mathrm i\}$ is also fixed by $\tau$, the matrices
$\gamma_\beta D_j$ are fixed by $\tau$. Hence,
\[
\mathcal B_\rho
\widetilde E_\rho(t_j)^\tau
\mathcal B_\rho^{-1}
=
\widetilde E_\rho(t_j).
\]
The representation $E_{\rho,\lambda}$ is obtained by tensoring $\widetilde
E_\rho$ with the inflation of the ordinary Specht representation
$S^\lambda$. Since $S^\lambda$ may be realized over $\Q$, the operator $
\mathcal B_\rho\otimes I_{S^\lambda} $ intertwines $E_{\rho,\lambda}^\tau$
and $E_{\rho,\lambda}$.

If $E_{\rho,\lambda}$ is non-self-associate, this shows that its two
associate representations are fixed individually by $\tau$, and hence $
e_\tau(E_{\rho,\lambda})=0 $.
Suppose now that $E_{\rho,\lambda}$ is self-associate. By
Proposition~\ref{prop:asso1}, a normalized associator $J_{\rho,r,\lambda}$
is obtained from $ c_r\Omega_r$ with $\Omega_r=F_rF_{r-1}\cdots F_1$,
where $ (c_r\Omega_r)^2=I$. Since the matrices $F_\nu$ anticommute
pairwise and satisfy $F_\nu^2=I$, we have $ \Omega_r^2 =
(-1)^{r(r-1)/2}I$.
It follows that
\[
c_r^2=(-1)^{r(r-1)/2},
\]
and is therefore fixed by $\tau$. Then
$J_{\rho,r,\lambda}$ is fixed by $\tau$, and commutes with $\mathcal
B_\rho\otimes I_{S^\lambda}$. Hence, Lemma~\ref{lem:galois-index-two}
gives $e_\tau(E_{\rho,\lambda})=0$.
\end{proof}

\begin{theorem}
\label{thm:galois-level-beta}
Let $\beta\geq 1$, and let
\[
\mathcal C_\beta
=
\bigl(
\lambda_\beta^{(0)},\lambda_\beta^{(1)},\ldots,
\lambda_\beta^{(e_\beta)} 
\bigr) \in \mathfrak C_\beta
\]
label the association class containing $\rho_{\beta,\mathcal
C_\beta}$, or by $\rho_{\beta,\mathcal C_\beta}^\pm$ in the
non-self-associate case. Assume that $|\lambda_\beta^{(j)}|< p$ for all
$0\leq j\leq e_\beta$. Then this association class is stable under $\tau$,
and
\[
e_\tau(\mathcal C_\beta) \equiv \beta\,\ell\bigl(\lambda_\beta^{(0)}\bigr)
\mod2.
\]
\end{theorem}

\begin{proof}
We apply Lemma~\ref{lem:BN-propagation}, Lemma~\ref{lem:galois-D} and
Lemma~\ref{lem:galois-NSA-elementary}. These results give
\begin{equation}
\label{eq:p1}
e_\tau(\Psi_{\beta,\mathcal C_\beta})
\equiv
\beta\,\ell\bigl(\lambda_\beta^{(0)}\bigr)
\mod2.
\end{equation}
It remains to determine the action on the constituents of
\[
\rho_{\beta,\mathcal C_\beta}
=
\Ind_{K_\beta}^{H_\beta^+}
\Psi_{\beta,\mathcal C_\beta}.
\]
Since induction commutes with Galois conjugation and with tensoring by
the sign character, this preserves the Galois action.
Indeed, in the non-self-associate case, if $\tau$ fixes
$\Psi_{\beta,\mathcal C_\beta}$, respectively exchanges it with its
associate, then it fixes, respectively exchanges, the two corresponding
induced representations. 
Assume that $\Psi=\Psi_{\beta,\mathcal C_\beta}$ is self-associate. Write
$R=\rho_{\beta,\mathcal C_\beta}$ and assume it is realized on
$\mathbb C[H_\beta^+]\otimes_{\mathbb C[K_\beta]}V$
so that $ hk\otimes v=h\otimes\Psi(k)v$ for $h\in H_\beta^+,\ k\in
K_\beta,\ v\in V$ and $ R(g)(h\otimes v)=gh\otimes v$.
Let $J$ be a normalized associator for $\Psi$, so that
$
J\Psi(k)=\varepsilon_{K_\beta}(k)\Psi(k)J$ for $k\in K_\beta$.
Define
\[
\widetilde J(h\otimes v)
=
\varepsilon_{H_\beta^+}(h)\,h\otimes Jv,
\]
which is well defined because
\[
\begin{aligned}
\widetilde J(hk\otimes v)
&=
\varepsilon_{H_\beta^+}(hk)\,hk\otimes Jv
=
\varepsilon_{H_\beta^+}(h)\,
h\otimes\varepsilon_{K_\beta}(k)\Psi(k)Jv
=
\varepsilon_{H_\beta^+}(h)\,
h\otimes J\Psi(k)v\\
&=
\widetilde J(h\otimes\Psi(k)v).
\end{aligned}
\]
For $g\in H_\beta^+$, $h\in H_{\beta}^+$ and $v\in V$, we have
\[
\widetilde J R(g)(h\otimes v)
=
\widetilde J(gh\otimes v)
=
\varepsilon_{H_\beta^+}(gh)\,gh\otimes Jv
=
\varepsilon_{H_\beta^+}(g)\,
R(g)\widetilde J(h\otimes v).
\]
Hence,
$
\widetilde J R(g)\widetilde J^{-1}
=
\varepsilon_H(g)R(g)$,
and, since $J^2=I$, we also have $\widetilde J^2=I$. 
Now let $B$ intertwine $\Psi^\tau$ with $\Psi$, so that
$
B\Psi^\tau(k)=\Psi(k)B$ for $k\in K_\beta$.
We have 
$ R^\tau\simeq\Ind_{K_\beta}^{H_\beta^+}(\Psi^\tau)$.  For $h\in H_\beta^+$ and
$v\in V$, define $\widetilde B(h\otimes v)=h\otimes Bv$.
This is well defined because
\[
\widetilde B(hk\otimes v)
=
hk\otimes Bv
=
h\otimes\Psi(k)Bv
=
h\otimes B\Psi^\tau(k)v
=
\widetilde B(h\otimes\Psi^\tau(k)v).
\]
Furthermore, we have
\[
\widetilde B R^\tau(g)(h\otimes v)
=
gh\otimes Bv
=
R(g)\widetilde B(h\otimes v),
\]
so $\widetilde B$ intertwines $R^\tau$ with $R$.
Moreover, if
$
B J^\tau B^{-1}
=
(-1)^{e_\tau(\Psi_{\beta,\mathcal C_\beta})}J$,
then 
$$\widetilde{J}^\tau(h\otimes v)=
\varepsilon_{H_\beta^+}(h)h\otimes J^{\tau}(v).
$$
Hence,
\[
\widetilde B\,\widetilde J^\tau\,\widetilde B^{-1}
=
(-1)^{e_\tau(\Psi_{\beta,\mathcal C_\beta})}\widetilde J.
\]
Hence
$
e_\tau(\rho_{\beta,\mathcal C_\beta})
=
e_\tau(\Psi_{\beta,\mathcal C_\beta})
$
and \eqref{eq:p1} gives the result.
\end{proof}

\section{Proof of the main theorem}
\label{sec:main-proof}

Let $\lambda\in\mathcal D_n$, and write
$
\mathcal C_\beta(\lambda)
=
\bigl(
\lambda_\beta^{(0)},\lambda_\beta^{(1)},\ldots,
\lambda_\beta^{(e_\beta)}
\bigr)
$ 
for the $\beta$-th level of its $p$-bar core tower. Recall that
the $p$-adic expansion of $n$ is
\[
n=\sum_{\beta\geq0}n_\beta p^\beta.
\]
By the $p'$-degree criterion for the $p$-bar core tower
(see~\cite[Proposition~7.3]{olsson}), the spin character labelled by
$\lambda$ has degree prime to $p$ if and only if
\begin{equation}
\label{eq:pprime-tower}
\sum_{j=0}^{e_\beta}
\bigl|\lambda_\beta^{(j)}\bigr|
=
n_\beta
\qquad
\text{for every }\beta\geq0.
\end{equation}
For each bar partition $\lambda$ labelling a spin character of
$p'$-degree, let $\mathcal X_\lambda$ denote the association class of
spin characters of $\overline N$
obtained as follows. The factor $\widetilde{\sym}_{n_0}$ is labelled by
$\lambda_0^{(0)}$. For every $\beta\geq1$,
condition~\eqref{eq:pprime-tower} says precisely that
$\mathcal C_\beta(\lambda)$ is an admissible $\beta$-tuple, and hence
labels an association class of $H_\beta^+$ by
Theorem~\ref{thm:param-level}.
Via inflation, we also regard $\mathcal X_\lambda$ as an association class
of spin $p'$-characters of $N$. In particular, we obtain a bijection
\begin{equation}
\label{eq:bijass}
\lambda\longmapsto\mathcal X_\lambda
\end{equation}
between the association classes of spin $p'$-characters of
$\widetilde{\sym}_n$ and those of $N$.

\begin{lemma}
\label{lem:global-local-association}
The global association class labelled by $\lambda$ and the local
association class $\mathcal X_\lambda$ have the same association type.
\end{lemma}

\begin{proof}
The global class labelled by $\lambda$ is self-associate if and only if
\[
n-\ell(\lambda)\equiv0\mod2.
\]
For $\beta\geq1$, Lemma~\ref{lem:typeQi} and
Theorem~\ref{thm:param-level} show that the class labelled by
$\mathcal C_\beta(\lambda)$ is self-associate if and only if
\[
n_\beta-\ell\bigl(\lambda_\beta^{(0)}\bigr)
\equiv0\mod2.
\]
The same criterion holds for the factor
$\widetilde{\sym}_{n_0}$ labelled by $\lambda_0^{(0)}$.
By \cite[Proposition 1.2]{MichlerOlsson}, 
the association type of
$\mathcal X_\lambda$ is then determined by
\[
\sum_{\beta\geq0}
\left(
n_\beta-\ell\bigl(\lambda_\beta^{(0)}\bigr)
\right)
\mod2.
\]
Moreover, using \eqref{eq:expansionadicn}, we obtain
\[
n\equiv\sum_{\beta\geq0}n_\beta\mod2,
\]
because $p$ is odd. Furthermore, the parity relation for the $p$-bar core
tower gives
\[
\ell(\lambda)
\equiv
\sum_{\beta\geq0}
\ell\bigl(\lambda_\beta^{(0)}\bigr)
\mod2.
\]
Hence,
\[
n-\ell(\lambda)
\equiv
\sum_{\beta\geq0}
\left(
n_\beta-\ell\bigl(\lambda_\beta^{(0)}\bigr)
\right)
\mod2,
\]
and the result follows.
\end{proof}
\noindent
For a bar partition $\lambda$, we write $ e_{\tau_p}(\mathcal X_\lambda) $
for the integer (which is well-defined modulo $2$) such that
$c_{\tau_p}(\lambda)=(-1)^{e_{\tau_p}(\mathcal X_\lambda)}$, in the sense
of Lemma~\ref{lem:galois-index-two}.
\begin{lemma}
\label{lem:global-local-galois}
With the above notation, we have
\[
e_{\tau_p}(\mathcal X_\lambda)
=
e_{\tau_p}(\lambda).
\]
\end{lemma}

\begin{proof}
Since $n_0<p$, we have $ p\nmid |\widetilde{\sym}_{n_0}|$ and $ p\nmid
|\widetilde{\alt}_{n_0}|$. By Brauer's theorem~\cite[Theorem
(10.3)]{isaacs}, every irreducible character of these two groups is fixed
by $\tau_p$.
For $\beta\geq1$, Theorem~\ref{thm:galois-level-beta} gives
\[
e_{\tau_p}\bigl(\mathcal C_\beta(\lambda)\bigr)
\equiv
\beta\,\ell\bigl(\lambda_\beta^{(0)}\bigr)
\mod2.
\]
Lemma~\ref{lem:BN-propagation} then yields
\[
e_{\tau_p}(\mathcal X_\lambda)
\equiv
\sum_{\beta\geq1}
\beta\,\ell\bigl(\lambda_\beta^{(0)}\bigr)
\mod2.
\]
On the other hand, Theorem~\ref{thm:globalprat} gives
\[
e_{\tau_p}(\lambda)
\equiv
\nu_p(\pi_\lambda)
\equiv
\sum_{\beta\geq1}
\beta\,\ell\bigl(\lambda_\beta^{(0)}\bigr)
\mod2.
\]
The result follows.
\end{proof}

\begin{proof}[Proof of Theorem~\ref{thm:main}]
For each $\lambda$, if $\mathcal X_\lambda$ is self-associate, we write $
\mathcal X_\lambda=\{\psi_\lambda\}$. If $\mathcal X_\lambda$ is
non-self-associate, we write $ \mathcal
X_\lambda=\{\psi_\lambda^+,\psi_\lambda^-\}$ with
$\psi_\lambda^-=\varepsilon_N\psi_\lambda^+$, where $\varepsilon_N$
denotes the restriction of the sign character to $N$. 
Using Lemma~\ref{lem:global-local-association} and
bijection~\eqref{eq:bijass}, we obtain a bijection
\[
\Omega:
\Irr_{p'}^{\mathrm{spin}}(\widetilde{\sym}_n)
\longrightarrow
\Irr_{p'}^{\mathrm{spin}}(N)
\]
satisfying $\Omega(\xi_\lambda)=\psi_\lambda$ and
$\Omega(\xi_\lambda^\pm)=\psi_\lambda^{\pm}$.
By Lemma \ref{lem:global-local-galois}, $\Omega$ is $\tau_p$-equivariant.
The result follows.
\end{proof}

Finally, the correspondence descends to the double cover of the
alternating group.
\begin{corollary}
Let $P\in\Syl_p(\widetilde A_n)$. The bijection of
Theorem~\ref{thm:main} induces a $G(p)$-equivariant bijection
\[
\Irr_{p'}^{\mathrm{spin}}(\widetilde A_n)
\longrightarrow
\Irr_{p'}^{\mathrm{spin}}
\bigl(N_{\widetilde A_n}(P)\bigr).
\]
\end{corollary}

\begin{proof}
Since $p$ is odd, $P$ is also a Sylow $p$-subgroup of
$\widetilde S_n$, and
\[
N_{\widetilde A_n}(P)
=
N_{\widetilde S_n}(P)\cap\widetilde A_n.
\]
By Lemma~\ref{lem:global-local-association}, the bijection $\Omega$
preserves association type. Clifford theory therefore yields a bijection
between the corresponding spin characters of the index-two subgroups,
which is $G(p)$-equivariant by
Lemma~\ref{lem:global-local-galois}.    
\end{proof}

\section{A blockwise version}
\label{sec:blockwise}

We now derive a blockwise version of the preceding correspondence.
Let $B$ be a spin $p$-block of $\tSym_n$, with $p$-bar core $\gamma$
and $p$-bar weight $w$, so that
\[
n=|\gamma|+pw.
\]
Let $D$ be a defect group of $B$ and let $b$ be the Brauer correspondent
of $B$ in $N_{\tSym_n}(D)$. We assume that $w>0$, the defect-zero case
being immediate.
Write the $p$-adic expansion 
\begin{equation}
\label{eq:block-p-adic}
pw=\sum_{\beta\geq1}n_\beta p^\beta,
\end{equation}
with $ 0\leq n_\beta<p$.
In this section, admissible $\beta$-tuples are understood with respect
to the coefficients $n_\beta$ in \eqref{eq:block-p-adic}.
By \cite[Theorem~13.1 and Proposition~13.8]{olsson}, the height-zero
spin characters in $B$ are parametrized by the $p$-bar core towers
$
\bigl(
\gamma;
\mathcal C_1,\mathcal C_2,\ldots
\bigr)$, 
where $\mathcal C_\beta\in\mathfrak C_\beta$ for every $\beta\geq1$.
By \cite[Proposition~13.3]{olsson}, $D$ may be chosen as a Sylow
$p$-subgroup of $\tSym_{pw}$. Put $r=|\gamma|$. Then
\[
N_{\tSym_n}(D)
\simeq
\tSym_r
\widehat\times
N_{\tSym_{pw}}(D),
\]
and \eqref{eq:reduced-local-decomposition} gives
\[
N_{\tSym_{pw}}(D)/D'
\simeq
\widehat{\prod}_{\beta\geq1}H_\beta^+.
\]
Hence, the height-zero association classes of $b$ are also parametrized by
$
\bigl(
\gamma;
\mathcal C_1,\mathcal C_2,\ldots
\bigr)
$
with $\mathcal C_\beta\in\mathfrak C_\beta$. For such a parameter $q$, let
$\mathcal X_q$ denote the corresponding local association class.

\begin{lemma}
\label{lem:block-same-type-bit}
The global and local association classes corresponding to the same
parameter
$
q=
\bigl(
\gamma;
\mathcal C_1,\mathcal C_2,\ldots
\bigr)
$
have the same association type and the same Galois exponent.
\end{lemma}

\begin{proof}
Let $\lambda$ be the bar partition corresponding to $q$. By
Lemma~\ref{lem:typeQi} and \cite[Proposition 1.2]{MichlerOlsson}, the
local association type is determined by
\[
|\gamma|-\ell(\gamma)
+
\sum_{\beta\geq1}
\left(
n_\beta-\ell(\lambda_\beta^{(0)})
\right)
\mod2.
\]
Since $p$ is odd,
\[
n
\equiv
|\gamma|+\sum_{\beta\geq1}n_\beta
\mod2.
\]
Moreover, since $\lambda_0^{(0)}=\gamma$,
applying \eqref{eq:length-parity-tower} with $k=0$, we obtain
\[
\ell(\lambda)
\equiv
\ell(\gamma)
+
\sum_{\beta\geq1}
\ell\bigl(\lambda_\beta^{(0)}\bigr)
\pmod2.
\]
Thus, the local type is equal to $ n-\ell(\lambda)\mod2 $, which is equal
to the global association type.

For the Galois action, the core factor has exponent zero. Indeed,
$\gamma$ is a $p$-bar core, and Theorem~\ref{thm:globalprat} gives
$
e_{\tau_p}(\gamma)=0$.
For $\beta\geq1$, Theorem~\ref{thm:galois-level-beta} gives
$
e_{\tau_p}(\mathcal C_\beta)
\equiv
\beta\,\ell(\lambda_\beta^{(0)})
\mod2$.
Hence, Lemma~\ref{lem:BN-propagation} yields
\[
e_{\tau_p}(\mathcal X_q)
\equiv
\sum_{\beta\geq1}
\beta\,\ell(\lambda_\beta^{(0)})
\pmod2.
\]
By Theorem~\ref{thm:globalprat}, the right-hand side is also the
Galois exponent of the global class labelled by $\lambda$.
\end{proof}

\begin{theorem}
\label{thm:blockwise-main}
Let $B$ be a spin $p$-block of $\tSym_n$, let $D$ be a defect group of
$B$, and let $b$ be its Brauer correspondent in $N_{\tSym_n}(D)$.
Then there exists an association-compatible $G(p)$-equivariant
bijection
\[
\Omega_B:
\Irr_0^{\mathrm{spin}}(B)
\longrightarrow
\Irr_0^{\mathrm{spin}}(b).
\]
In particular, the numbers of $p$-rational height-zero spin characters of
$B$ and of $b$ are equal.
\end{theorem}

\begin{proof}
The global and local height-zero association classes are parametrized by
the same collections $ \bigl( \gamma; \mathcal C_1,\mathcal C_2,\ldots
\bigr)$ with $\mathcal C_\beta\in\mathfrak C_\beta$ for every
$\beta\geq1$.
By Lemma~\ref{lem:block-same-type-bit}, corresponding classes have the
same association type and the same Galois exponent. Identifying equal
parameters therefore gives an association-compatible
$\tau_p$-equivariant bijection, and hence a $G(p)$-equivariant
bijection. 
\end{proof}

\begin{corollary}
\label{cor:blockwise-alternating}
Let $B_A$ be a spin $p$-block of $\tAlt_n$, let $D$ be a defect
group of $B_A$, and let $b_A$ be its Brauer
correspondent in $N_{\tAlt_n}(D)$. Then there exists a
$G(p)$-equivariant bijection
\[
\Irr_0^{\mathrm{spin}}(B_A)
\longrightarrow
\Irr_0^{\mathrm{spin}}(b_A).
\]
\end{corollary}

\begin{proof}
Let $B$ be the spin $p$-block of $\tSym_n$ corresponding to $B_A$.
Since $p$ is odd, $B$ and $B_A$ have the same defect group $D$.
Moreover, by~\cite[Lemma~2.3]{MichlerOlsson}, their Brauer
correspondents $b$ and $b_A$ are corresponding blocks.

Since the bijection of Theorem~\ref{thm:blockwise-main} preserves
association type, Clifford theory induces the corresponding bijection for
the index-two subgroups. Its $G(p)$-equivariance follows from
Lemmas~\ref{lem:block-same-type-bit} and~\ref{lem:galois-index-two}.
\end{proof}

{\bf Acknowledgements.} The first author acknowledges the support of the
ANR project CORTIPOM ANR-21-CE40-0019. The second author also acknowledges
the support of PSC CUNY grant GR-00017903 . 
The authors would like to thank the American Institute of Mathematics
(AIM) for its generous support and hospitality during a SQuaRE program,
\emph{Towards the Galois-McKay Conjecture}, where substantial progress on
this project was made. 

\bibliographystyle{abbrv}
\bibliography{references_19_1}

\end{document}